\documentclass{amsart}

\usepackage{amscd,amsmath,amssymb,amsfonts,bbm, calligra, xspace}
\usepackage[T1]{fontenc}
\usepackage{lmodern}
\usepackage{mathtools}
\usepackage{enumerate, enumitem}
\usepackage{multicol}
\usepackage[all]{xy}
\usepackage{mathrsfs}
\usepackage{footmisc}
\usepackage{yhmath}

\usepackage{xcolor}
\definecolor{darkblue}{rgb}{0,0,0.8}
\definecolor{darkgreen}{rgb}{0,0.4,0}
\usepackage[colorlinks=true,linkcolor=darkblue, citecolor=darkgreen,  urlcolor=darkgreen, unicode,pdfborder={0 0 0},final]{hyperref}

\newtheorem{thm}{Theorem}[section]
\newtheorem{prop}[thm]{Proposition}

\newtheorem{lem}[thm]{Lemma}
\newtheorem{cor}[thm]{Corollary}

\theoremstyle{definition}

\newtheorem{quest}[thm]{Question}

\theoremstyle{remark}
\newtheorem{rem}[thm]{Remark}
\newtheorem{rems}[thm]{Remarks}

\numberwithin{equation}{section}

\newcommand{\Proj}{\mathrm{Proj}}
\newcommand{\HF}{\mathrm{HF}}
\newcommand{\Id}{\mathrm{Id}}
\newcommand{\Eff}{\mathrm{Eff}}
\newcommand{\Nef}{\mathrm{Nef}}

\newcommand{\Cox}{\mathrm{Cox}}
\newcommand{\Ker}{\mathrm{Ker}}
\newcommand{\Hilb}{\mathrm{Hilb}}
\newcommand{\Ima}{\mathrm{Im}}
\newcommand{\Hom}{\mathrm{Hom}}
\newcommand{\SL}{\mathrm{SL}}
\newcommand{\Spec}{\mathrm{Spec}}
\newcommand{\Pic}{\mathrm{Pic}}
\newcommand{\Cl}{\mathrm{Cl}}
\newcommand{\Grass}{\mathrm{Grass}}

\newcommand{\isoto}{\myxrightarrow{\,\sim\,}}
\makeatletter
\def\myrightarrow{{\setbox\z@\hbox{$\rightarrow$}\dimen0\ht\z@\multiply\dimen0 6\divide\dimen0 10\ht\z@\dimen0\box\z@}}
\def\myrightarrowfill@{\arrowfill@\relbar\relbar\myrightarrow}
\newcommand{\myxrightarrow}[2][]{\ext@arrow 0359\myrightarrowfill@{#1}{#2}}
\makeatother

\makeatletter
\newcommand{\extp}{\@ifnextchar^\@extp{\@extp^{\,}}}
\def\@extp^#1{\mathop{\bigwedge\nolimits^{\!#1}}}
\makeatother

\def\loccit{\emph{loc}.\kern3pt \emph{cit}.{}\xspace}
\def\eg{e.g.\kern.3em}
\def\ie{i.e.,\ }
\def\resp {\text{resp.}\kern.3em}

\def\A{\mathbb A}

\def\Z{\mathbb Z}
\def\C{\mathbb C}

\def\G{\mathbb G}

\def\Q{\mathbb Q}
\def\P{\mathbb P}
\def\R{\mathbb R}

\def\lll{\langle}
\def\rr{\rangle}
\def\;{\hspace{.05em}}

\def\cM{\mathcal{M}}
\def\cL{\mathcal{L}}
\def\cO{\mathcal{O}}
\def\cE{\mathcal{E}}
\def\cF{\mathcal{F}}
\def\cG{\mathcal{G}}
\def\cH{\mathcal{H}}

\def\cM{\mathcal{M}}

\def\cQ{\mathcal{Q}}

\def\cV{\mathcal{V}}
\def\cW{\mathcal{W}}
\def\cY{\mathcal{Y}}
\def\cX{\mathcal{X}}
\def\cZ{\mathcal{Z}}

\def\whE{\widehat{E}}

\def\whP{\widehat{P}}

\def\tp{\tilde{p}}
\def\tq{\tilde{q}}

\def\oP{\overline{P}}

\def\wQ{\widetilde{Q}}

\def\wW{\widetilde{W}}

\def\wDelta{\widetilde{\Delta}}
\def\wOmega{\widetilde{\Omega}}

\def\wmu{\widetilde{\mu}}

\def\whmu{\widehat{\mu}}

\begin{document}

\title[The resultant divisor is negative]
{The resultant divisor is negative}

\author{Olivier Benoist}
\address{D\'epartement de math\'ematiques et applications, \'Ecole normale sup\'erieure, CNRS,
45 rue d'Ulm, 75230 Paris Cedex 05, France}
\email{olivier.benoist@ens.fr}

\renewcommand{\abstractname}{Abstract}
\begin{abstract}
Fix two integers $1\leq d<e$. 
We study the birational geometry of a parameter space for pairs of homogeneous polynomials of degrees $d$ and $e$ in two variables (in which the higher degree polynomial is well defined only up to a multiple of the lower degree polynomial). We show that one can run the MMP on this space, and that it eventually contracts the resultant divisor.
\end{abstract}

\maketitle

\section{Introduction}\label{intro}

Studying the birational geometry of moduli or parameter spaces, especially by running the minimal model program (MMP) on them, is now a classical theme in algebraic geometry. The archetypal example is the Hassett--Keel program for running the MMP on the moduli space $\overline{\cM}_g$ of stable curves of genus $g$, which has recently been completed for $g\leq 4$ (see \cite{HH1,HH2, CMJL2, AFS3, ADLW}).

In this article, we consider and solve a similar problem for pairs of homogeneous polynomials of distinct degrees  in two variables. We prove that one can run the MMP on an appropriate parameter space for such pairs, and that this MMP eventually contracts the resultant divisor parameterizing noncoprime polynomials. This demonstrates the existence of projective curves avoiding the resultant divisor.

\subsection{The resultant divisor}
\label{parresultant}

Let $k$ be a field. For $l\in\Z$, let~$V_l:=H^0(\P^1_k,\cO_{\P^1_k}(l))$ be the space of degree $l$ homogeneous polynomials in two variables. Let~${P_l:=\P(V_l)}$ be its projectivization.  We let $\lll F\rr\in P_l$ be the point induced by a nonzero~$F\in V_l$. For~$F\in V_l$ and $m\in\Z$, let $\lll F\rr_m\subset V_m$ be the space of degree $m$ multiples of~$F$.  

 Fix two integers $1\leq d<e$.  Let $\wDelta_{d,e}\subset P_{d}\times P_{e}$ be the subvariety parameterizing those $(\lll F\rr, \lll G\rr)$ such that $F$ and $G$ are not coprime (\ie such that~$F$ and $G$ have a common zero on $\P^1_k$). This subvariety is the zero locus of the classical \textit{resultant polynomial}~$R_{d,e}\in H^0(P_d\times P_e,\cO_{ P_d\times P_e}(e,d))$ (a polynomial in the coefficients of~$F$ and~$G$ vanishing exactly when $F$ and $G$ are not coprime). It follows that~$\wDelta_{d,e}$ is an ample divisor in $P_d\times P_e$. In particular, any projective curve in~$P_d\times P_e$ meets~$\wDelta_{d,e}$.

The purpose of this article is to show that, in spite of this, the resultant has strong negativity properties. It turns out that $P_d\times P_e$ is not the optimal parameter space to consider, because the coprimality of a pair ${(\lll F\rr, \lll G\rr)\in P_d\times P_e}$ depends on~$G$ only up to a multiple of $F$. To take this into account, we introduce the projective bundle~$P_{d,e}\to P_d$  whose fiber over $\lll F\rr$ is the projectivization $\P(V_e/\lll F\rr_e)$ of the space of degree $e$ polynomials up to a multiple of $F$. The point of $P_{d,e}$ associated with a nonzero $G\in V_e/\lll F\rr_e$ will be denoted by~$[F,G]$. The divisor~$\wDelta_{d,e}$ induces a divisor~$\Delta_{d,e}\subset P_{d,e}$, which we call the \textit{resultant divisor}. Set~$U_{d,e}:=P_{d,e}\setminus \Delta_{d,e}$.

The most striking consequence of our results is the following theorem:

\begin{thm}[Corollary \ref{corcompactification}]
\label{thmain}
The variety $U_{d,e}$ admits a projective compactification $X_{d,e}$ whose boundary $X_{d,e}\setminus U_{d,e}$ has codimension $2$ in $X_{d,e}$.
\end{thm}

Taking appropriate linear sections of $X_{d,e}$, we deduce the following.

\begin{thm}[Theorem \ref{completecurves}]
\label{corcurves}
The variety $U_{d,e}$ is covered by projective curves.
\end{thm}

We do not know how to construct projective curves avoiding the resultant divisor~$\Delta_{d,e}$ directly (except in particular cases, \eg when $e=d+1$, see Remark~\ref{explicurve}). Let us emphasize the concrete meaning of Theorem \ref{corcurves}. It states that one can find many one-parameter algebraic families of pairs $[F,G]$ of polynomials of degrees $d$ and $e$, such that $F$ and~$G$ do not acquire a common root under any degeneration.

\subsection{Running the MMP}

Let us now explain the strategy of our proof of Theorem~\ref{thmain}.
To construct $X_{d,e}$ from~$P_{d,e}$, we need to somehow contract the resultant divisor $\Delta_{d,e}$. However, the birational map  $P_{d,e}\dashrightarrow X_{d,e}$ is not a morphism in general: the divisor $\Delta_{d,e}$ can only be contracted after an appropriate birational modification of $X_{d,e}$. The best one can hope for is to perform a series of flips (birational modifications in codimension~$\geq 2$) on $P_{d,e}$, after which it may become possible to contract the strict transform of $\Delta_{d,e}$. 

Such a sequence of flips should be the outcome of running the MMP on $P_{d,e}$. As~$P_{d,e}$ has Picard rank $2$ (except when $d=1$, see~\S\ref{parPicard}), there are two directions in which one could run the MMP (this is a so-called $2$-ray game, see \eg \hbox{\cite[\S2.2]{Corti}} or \cite[\S2.F]{BLZ}). One of the two directions yields the projection~\mbox{$P_{d,e}\to P_d$}, and it is the other direction that could eventually contract the resultant divisor.

The MMP cannot be applied to all smooth projective varieties. In our situation, it is possible to run it successfully on $P_{d,e}$ if and only if $P_{d,e}$ is a Mori dream space in the sense of \cite[Definition 1.10]{MDS} (see \cite[Proposition 1.11]{MDS}). Our main theorem states that this is the case, gives a full description of the loci flipped by the MMP applied to $P_{d,e}$, and shows that the divisor $\Delta_{d,e}$ eventually gets contracted.

\begin{thm}[Theorem \ref{thMMPmain}]
\label{thMDS}
The following assertions hold.
\begin{enumerate}[label=(\roman*)] 
\item
\label{i}
The variety $P_{d,e}$ is a Mori dream space.
\item
\label{ii}
The effective cone $\Eff(P_{d,e})$ of $P_{d,e}$ is generated by the pullback $\cO_{P_{d,e}}(1,0)$ of $\cO_{P_d}(1)$ to $P_{d,e}$ and by the line bundle $\cO_{P_{d,e}}(\Delta_{d,e})$. 
\item
\label{iii}
The MMP for $P_{d,e}$ flips $W_{d,e}^i:=\{[F,G]\in P_{d,e}\mid\deg(\gcd(F,G))\geq d-i\}$ for~$1\leq i\leq d-2$ and eventually contracts $W_{d,e}^{d-1}=\Delta_{d,e}$.
\item
\label{iv}
 The output of the MMP for $P_{d,e}$ is a projective compactification $X_{d,e}$ of~$U_{d,e}$, whose boundary has codimension $2$, and that admits a stratification whose strata are isomorphic to $(U_{d-j,e+j})_{0\leq j<d}$.
\end{enumerate}
\end{thm}

The low dimension and inductive structure of the boundary of $X_{d,e}$ is reminiscent of Baily--Borel compactifications of hermitian locally symmetric spaces.

\subsection{Constructing birational models}
\label{parintrobir}

We take a concrete approach to Theorem~\ref{thmain}, giving direct constructions of the varieties appearing in the MMP for~$P_{d,e}$. 

In Section \ref{secblowup}, we introduce the variety $\whP_{d,e}$ obtained from~$P_{d,e}$ by blowing up successively the strict transforms of the $(W_{d,e}^u)_{1\leq u<d}$. The strict transform ${\whE_{d,e}^u\subset \whP_{d,e}}$ of the exceptional divisor of the $u$-th blow-up is isomorphic to $\whP_{d-u,e+u}\times\whP_{u,e-d+u}$ (see Theorem \ref{blowup} \ref{bii} and Remark \ref{blowupremarks} (i)). This surprising structural result for $\whE_{d,e}^u$ allows for arguments by induction on $d$ that are at the heart of our proofs.

Associating with $[F,G]\in U_{d,e}$ the~$\G_m$\nobreakdash-invariant subscheme~$\{F=G=0\}$ of $\A^2_k$ yields a rational map~$\rho_{d,e}:\whP_{d,e}\dashrightarrow \Hilb_{d,e}$ to a multigraded Hilbert scheme~$\Hilb_{d,e}$ in the sense of~\cite{multigraded} (see \S\ref{parmultigraded}). In Section \ref{secmultigraded}, we study the variety~$\Hilb_{d,e}$ in detail. For each~$l\geq 0$, sending~$Z\in\Hilb_{d,e}$ to the space of degree $l$ equations of $Z$ yields a morphism from~$\Hilb_{d,e}$ to a Grassmannian, hence to a semi-ample bundle on~$\Hilb_{d,e}$ which we compute (see Proposition \ref{Licalcul}). Using these semi-ample line bundles and leveraging our excellent understanding of the geometry of~$\whP_{d,e}$, we deduce that the rational map $\rho_{d,e}:\whP_{d,e}\dashrightarrow \Hilb_{d,e}$ is an isomorphism (see Theorem \ref{thHilb}).

The pull-backs by $\rho_{d,e}$ of our semi-ample line bundles give rise to birational models of $P_{d,e}$ constructed as contractions of $\whP_{d,e}$. In Section \ref{secMMP}, we study these birational models and we show that they realize the MMP for $P_{d,e}$ (see \eqref{MMPdiag}). This leads to a proof of Theorem \ref{thMDS}.


\subsection{The Mori dream space property and invariant theory}
\label{parintroMDS}

More conceptual criteria exist for determining whether a given smooth projective variety is a Mori dream space. We have been unable to use them to establish Theorem~\ref{thMDS}~\ref{i}. 

First, any log Fano variety (say, over $k=\C$) is a Mori dream space by \cite[Corollary 1.3.2]{BCHM}.
We have been unable to use such a criterion to prove Theorem~\ref{thMDS}~\ref{i}. This is related to the fact that we need to run the MMP backwards: the small contractions to be flipped contract $K$-positive instead of $K$-negative curves.

Second, a smooth projective variety $X$ (say, with torsion-free Picard group) is a Mori dream space
if and only if its Cox ring $\Cox(X):=\bigoplus_{\cL\in\Pic(X)} H^0(X,\cL)$ is finitely generated (see \cite[Proposition 2.9]{MDS}). The Cox ring of~$P_{d,e}$ is easy to describe. Recall that $V_l=H^0(\P^1_k,\cO_{\P^1_k}(l))$ for $l\geq 0$. Let the additive group~$V_{e-d}$ act on the affine algebraic variety~$V_d\times V_e$ by~$H\cdot(F,G)=(F,G+HF)$. At least when~${d>1}$ (see Remark \ref{remCox}), the natural rational map~${V_d\times V_e\dashrightarrow P_{d,e}}$ identifies~$\Cox(P_{d,e})$ with the ring of invariants~${\cO(V_d\times V_e)^{V_{e-d}}}$. This gives a reformulation of Theorem~\ref{thMDS}~\ref{i} in the framework of Hilbert's $14$th problem. 

We have been unable to use this criterion to prove Theorem~\ref{thMDS}~\ref{i}. Instead, we deduce from Theorem~\ref{thMDS}~\ref{i} the following result pertaining to invariant theory.

\begin{thm}[Theorem \ref{thmCox}]
\label{thCox}
The $k$-algebra of invariants of the action of $V_{e-d}$ on~${V_d\times V_e}$ given by $H\cdot(F,G)=(F,G+HF)$ is finitely generated.
\end{thm}

Algebras of invariants of reductive group actions are always finitely generated. Theorem \ref{thCox} features an additive group action, for which no general result is available (see Nagata's counterexample to Hilbert's $14$th problem \cite{Nagata}).

\subsection{A comparison with \texorpdfstring{$(\P^2_k)^{[n]}$}{P2[n]}}

We briefly compare the birational geometry of~$P_{d,e}$ and of the Hilbert scheme $(\P^2_k)^{[n]}$ of length $n$ subschemes of~$\P^2_k$ (when $k=\C$). These varieties are related because the modification~$\whP_{d,e}$ of $P_{d,e}$ introduced in \S\ref{parintrobir} can be realized as a closed subvariety of~$(\P^2_k)^{[de]}$ (see Theorem~\ref{thHilb} and Remark~\ref{remHilb}).

Unlike $P_{d,e}$ in general, the variety $(\P^2_k)^{[n]}$ is always log Fano (see \cite[Theorem 2.5]{ABCH}).
It therefore follows from \cite[Corollary 1.3.2]{BCHM} that $(\P^2_k)^{[n]}$ is a Mori dream space. As explained in \S\ref{parintroMDS}, we are not aware of such a direct argument to prove Theorem \ref{thMDS} \ref{i} according to which $P_{d,e}$ is a Mori dream space.

As $(\P^2_k)^{[n]}$ has Picard rank $2$ by \cite{Fogarty2}, one can run its MMP in two directions. One of the two directions yields a divisorial contraction: the Hilbert--Chow morphism \cite[p.\,516]{Fogarty1}. It is the other direction that is interesting (it is described in great detail in \cite[\S10]{ABCH} when $n\leq 9$, see \cite[Theorem 1.2]{LZ} for general $n$).

Finally, the nontrivial boundary of $\Eff(P_{d,e})$ is generated by the resultant divisor~$\Delta_{d,e}$  (see Theorem \ref{thMDS} \ref{ii}). In contrast, the nontrivial boundary of $\Eff((\P^2_k)^{[n]})$ is difficult to describe and depends on $n$ in a complicated fashion (see~\cite[Theorem~1.4 and Table 1 p.\,59]{Huizenga}).

\subsection{The higher-dimensional case}
\label{parintrohigher}

The questions that are solved in this article for homogeneous polynomials in two variables are also interesting in more variables.

Fix $n\geq 1$ and $1\leq d<e$. Set $P^{(n)}_d:=\P(H^0(\P^n_k,\cO_{\P^n_k}(d)))$. Let ${P^{(n)}_{d,e}\to P^{(n)}_{d}}$ be the projective bundle whose fiber over
$\lll F\rr$ is~$\P(H^0(\P^n_k,\cO(e))/\lll F\rr_e)$. As above, we use the notation $[F,G]$ for points of $P^{(n)}_{d,e}$. We let $\Delta^{(n)}_{d,e}\subset P^{(n)}_{d,e}$ be the \textit{discriminant divisor} parameterizing those $[F,G]\in P^{(n)}_{d,e}$ such that $\{F=G=0\}$ is not a smooth complete intersection of codimension $2$ in $\P^n_k$.
Define $U^{(n)}_{d,e}:=P^{(n)}_{d,e}\setminus \Delta^{(n)}_{d,e}$.

When $n=1$, the discriminant divisor is exactly the resultant divisor. It therefore makes sense to ask if our main results admit generalizations for arbitrary $n\geq 1$.

\begin{quest}
\label{q1}
Is some Zariski-open subset~${V\subset U^{(n)}_{d,e}}$ covered by projective curves? 
\end{quest}

\begin{quest}
\label{q3}
Is the variety $P^{(n)}_{d,e}$ a Mori dream space?
\end{quest}

\begin{quest}
\label{q4}
Is the divisor $\Delta^{(n)}_{d,e}$ on the boundary of the effective cone of $P^{(n)}_{d,e}$?
\end{quest}

A more optimistic version of Question \ref{q1} would require that~$V=U^{(n)}_{d,e}$. 

When $n=3$ and $(d,e)=(2,3)$, Questions \ref{q3} and \ref{q4} were first considered in~\cite[\S1.3]{CMJL2} in relation with the Hassett--Keel program in genus $4$.

In the direction of Question \ref{q4}, it is known that the divisor $\Delta^{(n)}_{d,e}$ is never ample (see \cite[Remarque~2.9]{qp}).

Positive answers to Questions \ref{q3} and \ref{q4} would give rise to a positive answer to Question \ref{q1}.
To see it, use Question \ref{q3} to run the MMP on~$P^{(n)}_{d,e}$. If the last step of this MMP is a divisorial contraction, then Question \ref{q4} implies that it contracts (the strict transform of) $\Delta_{d,e}^{(n)}$. Curves constructed as general linear sections of the resulting variety then answer Question \ref{q1}. If the last step of this MMP is a fibration, then Question \ref{q4} implies that this fibration is induced by the line bundle associated with (the strict transform of) $\Delta_{d,e}^{(n)}$. Curves constructed as linear sections in a general fiber of this fibration then answer Question \ref{q1}.

Questions \ref{q1} to \ref{q4} have positive answers when $n=1$ (see Theorem \ref{thMDS}) and when~${d=1}$ and $n\geq 2$ (the MMP for $P_{d,e}^{(n)}$ can then be realized as a variation of GIT in the sense of \cite{VGIT} for the diagonal action of $\SL_n$ on~$P^{(n-1)}_e\times\P(M_{n+1,n})$).

When $d\geq 2$, we expect that the last step of the MMP for $P^{(n)}_{d,e}$ (whose existence is predicted by Question \ref{q3}) is a divisorial contraction (contracting the discriminant divisor in view of Question \ref{q4}). This would answer positively the next question.

\begin{quest}
\label{q2}
Assume that $d\geq 2$.
Does some Zariski-open subset $V\subset U^{(n)}_{d,e}$ admit a projective compactification whose boundary has codimension $\geq 2$? 
\end{quest}

Again, one could even hope that choosing $V=U^{(n)}_{d,e}$ in Question \ref{q2} works.

When $n=3$, a positive answer to Question \ref{q2} would show the existence of nonisotrivial complete families of smooth curves in $\P^3_k$, thereby solving a classical problem (see \cite{ChangRan1, ChangRan2}, \cite[p.\,57]{HM} and \cite{Complete}).

\subsection{Other values of the degrees}

One can consider further extensions of these questions. Fix degrees~${1\leq d_1\leq\dots\leq d_c}$. Let $U^{(n)}_{d_1,\dots,d_c}$ be the variety parameterizing those~$(F_1,\dots,F_c)$ in $\bigoplus_{i=1}^c H^0(\P^n_k,\cO_{\P^n_k},\cO(d_i))$ such that $\{F_1=\dots=F_c=0\}$ is smooth of codimension $c$ in $\P^n_k$, where $F_i$ is well defined up to multiples of the~$F_j$ of degree $\leq d_i$ and up to scalar multiplication. (The variety $U^{(n)}_{d_1,\dots,d_c}$ can be realized as an open subset of a multigraded Hilbert scheme parameterizing $\G_m$-invariant subschemes of~$\A^{n+1}_k$, by adapting the construction explained in \S\ref{parintrobir}.)

When the $d_i$ are all equal with common value $d$ (\eg in the ${c=1}$ case), the variety $U^{(n)}_{d_1,\dots,d_c}$ is the complement of the discriminant locus in the Grassmannian~${\Grass\big(c, H^0(\P^n_k,\cO_{\P^n_k}(d))\big)}$. Unless $d=1$ and $c<n+1$, this locus is a nonempty divisor (see \cite[Remarque 1.1]{Degrees}), hence an ample divisor, so~$U^{(n)}_{d_1,\dots,d_c}$ is affine. In particular, the variety $U^{(n)}_{d_1,\dots,d_c}$ cannot contain projective curves. This contrasts with Theorem~\ref{corcurves}. As far as we know, whenever the degrees are not all equal, the variety $U^{(n)}_{d_1,\dots,d_c}$ might be covered by projective curves. 

The case $c=2$ and $d_1<d_2$  is particular because the variety $U^{(n)}_{d_1,d_2}$ admits a very simple compactification $P^{(n)}_{d_1,d_2}$ with which one can work. In general, we are not aware of the existence of such a concrete compactification.

\subsection{Acknowledgements}

Part of this work was done during a stay at University of Utah, where I benefited from excellent working
conditions. Particular thanks to Tommaso de Fernex for many discussions and his warm hospitality.

I thank Kristin DeVleming for a very interesting conversation on related topics.

\section{Generalities and notation}
\label{secgeneralities}

In this section, we introduce the variety $P_{d,e}$ which is our main object of study.

\subsection{Conventions}

We fix a field $k$. A \textit{variety} over $k$ is a separated scheme of finite type over $k$. A \textit{curve} is a variety of pure dimension~$1$~over~$k$.

\subsection{The variety \texorpdfstring{$P_{d,e}$}{Pde}}

Set $V_l:=H^0(\P^1_k,\cO_{\P^1_k}(l))$ for $l\in\Z$. Let $P_l:=\P(V_l)$ be its projectivization. If~$F\in V_l$ and $m\in\Z$, we let $\lll F\rr_m\subset V_m$ be the space of multiples of~$F$ of degree $m$. If $F$ is nonzero, we write $\lll F\rr$ instead of~$\lll F\rr_l$ for the line spanned by $F$, which we view as a point of $P_l$.

For $1\leq d<e$, the vector bundle $\cE_{d,e}$ over $P_d$ whose fiber over~$\lll F\rr$ is $V_e/\lll F\rr_e$ fits into a short exact sequence
\begin{equation}
\label{sesqp}
0\to V_{e-d}\otimes\cO_{P_{d}}(-1)\xrightarrow{\cdot F} V_e\otimes\cO_{P_{d}}\to\cE_{d,e}\to0
\end{equation}
(see \eg \cite[(1) p.\,1753]{qp}). Let $P_{d,e}\to P_{d}$ be the associated projective bundle.
We let $[F,G]\in P_{d,e}$ be the point associated with nonzero $F\in V_d$ and $G\in V_e/\lll F\rr_e$.

Let $\Delta_{d,e}\subset P_{d,e}$ be the subset of all $[F,G]\in P_{d,e}$ such that $F$ and $G$ are not coprime. It is an irreducible divisor, which is empty if and only if $d=1$.

\subsection{Line bundles on \texorpdfstring{$P_{d,e}$}{Pde}}
\label{parPicard}

The Picard group $\Pic(P_{d,e})$ of $P_{d,e}$ is generated by the pull-back $\cO_{P_{d,e}}(1,0)$ of the ample tautological bundle on $P_{d}$ and by the relatively ample tautological bundle~$\cO_{P_{d,e}}(0,1)$ of the projective bundle~$P_{d,e}\to P_{d}$. All line bundles on $P_{d,e}$ are therefore of the form~$\cO_{P_{d,e}}(l,m)$ for some $l,m\in\Z$.
  
  Unless $d=1$, the line bundles $\cO_{P_{d,e}}(1,0)$ and  $\cO_{P_{d,e}}(0,1)$ are linearly independent, and the smooth projective variety~$P_{d,e}$ has Picard rank $2$. 
  
If $d=1$, the morphism ${P_{1,e}\to P_{1}\simeq \P^1_k}$ is an isomorphism, 
so~$P_{1,e}\simeq \P^1_k$ has Picard rank $1$. To compute the relation between $\cO_{P_{1,e}}(1,0)$ and~$\cO_{P_{1,e}}(0,1)$ in this case, note that $\cO_{P_{1,e}}(0,-1)\simeq\cE_{1,e}$ and take the determinant of \eqref{sesqp} to show that
\begin{equation}
\label{coincidence}
\cO_{P_{1,e}}(0,-1)\simeq\cE_{1,e}\simeq\cO_{P_{1,e}}(e,0)\textrm{ \hspace{.3em}in\hspace{.3em} }\Pic(P_{1,e}).
\end{equation}

\subsection{Multiplication maps}

For $0\leq u\leq d<e$, we define the multiplication map
\begin{equation}
\label{defphi}
\mu_{d,e}^u:P_{d-u}\times P_{u,e-d+u}\to P_{d,e}
\end{equation}
by $\mu_{d,e}^u(\lll A\rr ,[B,C])=[AB,AC]$. Let $W_{d,e}^u\subset P_{d,e}$ be the image of~$\mu_{d,e}^u$ endowed with its reduced structure. One has $W^0_{d,e}=\varnothing$ and $W^d_{d,e}=P_{d,e}$. Note that $W_{d,e}^{d-1}=\Delta_{d,e}$.

It will sometimes be easier to work on $P_d\times P_e$ rather than on $P_{d,e}$. For this reason, we introduce the morphisms $\widetilde{\mu}_{d,e}^u:P_{d-u}\times P_u\times P_{e-d+u}\to P_d\times P_e$
given by~$\widetilde{\mu}_{d,e}^u(\lll A\rr,\lll B\rr,\lll C\rr)=(\lll AB\rr,\lll AC\rr)$ and define $\wW_{d,e}^u$ to be the image of $\wmu_{d,e}^u$.

\section{Blowing up \texorpdfstring{$P_{d,e}$}{Pde}}
\label{secblowup}

In this section, we construct and describe a modification $\whP_{d,e}$ of $P_{d,e}$.

\subsection{The variety \texorpdfstring{$\whP_{d,e}$}{hatPde}}

Fix $1\leq d<e$. For $0\leq u<d$, we let $\beta^uP_{d,e}$ be the variety obtained  by blowing up in $P_{d,e}$ first $W_{d,e}^1$, then the strict transform of $W_{d,e}^2$, etc., and lastly the strict transform of~$W_{d,e}^{u}$.  We define $\whP_{d,e}:=\beta^{d-1}P_{d,e}$.

  For $0\leq u<v<d$, we let $\beta^uW_{d,e}^v$ be the strict transform of~$W_{d,e}^v$ in $\beta^uP_{d,e}$. We also let $\beta^u\mu_{d,e}^v:P_{d-v}\times\beta^u P_{v,e-d+v}\dashrightarrow\beta^u P_{d,e}$ be the rational map induced by~$\mu_{d,e}^v$ (see \eqref{defphi}). It will be shown to be a morphism in Proposition \ref{blowuprec} below. We define~$\whmu_{d,e}^{\,v}:=\beta^{v-1}\mu_{d,e}^v$.

For $1\leq u<d$, we let $\alpha^u_{d,e}:\beta^uP_{d,e}\to\beta^{u-1}P_{d,e}$ be the blow-up of~$\beta^{u-1}W^u_{d,e}$, with exceptional divisor $E_{d,e}^u\subset \beta^uP_{d,e}$. For $1\leq v\leq u<d$, we let $\beta^uE^v_{d,e}$ be the strict transform of $E^v_{d,e}$ in $\beta^u P_{d,e}$. We define ${\whE_{d,e}^u:=\beta^{d-1}E_{d,e}^u\subset \whP_{d,e}}$. 

When this cannot cause any confusion, we write $\cO_X(l,m)$ or $\cO(l,m)$ for the pull-back of the line bundle~$\cO_{P_{d,e}}(l,m)$ on any blow-up $X$ of $P_{d,e}$ (such as $X=\whP_{d,e}$).

Our goal  is the following theorem.

\begin{thm}
\label{blowup}
Fix $1\leq d<e$.
\begin{enumerate}[label=(\roman*)] 
\item
\label{bi}
The variety $\whP_{d,e}$ is smooth and the $(\whE_{d,e}^u)_{1\leq u<d}$ form a strict normal crossings divisor in it.
\item
\label{bii}
 For $1\leq u< d$, there is a canonical isomorphism 
 \begin{equation}
 \label{canoniso}
\varphi_{d,e}^u:\whP_{d-u,e+u}\times\whP_{u,e-d+u}\isoto \whE_{d,e}^u.
 \end{equation}
 Let $p_1:\whE_{d,e}^u\to\whP_{d-u,e+u}$ and $p_2:\whE_{d,e}^u\to \whP_{u,e-d+u}$ be the projections. 
\item
\label{biii}
 If $v<u$, then $\whE_{d,e}^v|_{\whE_{d,e}^u}=p_2^*\whE_{u,e-d+u}^v$ as Cartier divisors.
\item
\label{biv}
If $v>u$, then $\whE_{d,e}^v|_{\whE_{d,e}^u}=p_1^*\whE_{d-u,e+u}^{v-u}$ as Cartier divisors.
\item
\label{bv}
One has $\cO_{\whP_{d,e}}(\whE_{d,e}^u)|_{\whE_{d,e}^u}\simeq p_1^*\cO(1,-1)\otimes p_2^*\cO(1,1)(-\sum_{w=1}^{u-1}\whE_{u,e-d+u}^{w})$.
\item
\label{bvi}
One has $\cO_{\whP_{d,e}}(l,m)|_{\whE_{d,e}^u}\simeq p_1^*\cO(l+m,0)\otimes p_2^*\cO(l,m)$.
\end{enumerate}
\end{thm}

\subsection{The differential of the multiplication maps}

One of the difficulties of the proof of Theorem \ref{blowup} is that the $\mu_{d,e}^u$ become immersive only after the previous strata have been blown up. Proposition \ref{immdef2} will help us to overcome this obstacle.

\begin{prop}
\label{immdef2}
Fix $0\leq v\leq u<d<e$. Set $Q_{d,e}^u:=P_{d-u}\times P_{u,e-d+u}$.
\begin{enumerate}[label=(\roman*)] 
\item
\label{ddi}
The multiplication map $\mu^u_{d,e}:Q_{d,e}^u\to P_{d,e}$ is immersive on the open subset
$\Omega_{d,e}^u:=\{(\lll A\rr,[B,C])\in Q_{d,e}^u\mid\gcd(A,B,C)=1\}$.
\item
\label{ddii}
The normal bundle to $(\mu^u_{d,e})|_{\Omega_{d,e}^u}$ is isomorphic to $$(p_1^*\mathcal{E}_{d-u,e+u}(1)\otimes p_2^*\cO_{P_{u,e-d+u}}(1,1))|_{\Omega_{d,e}^u},$$ where~$p_1$ and $p_2$ are the projections of $Q_{d,e}^u$ onto $P_{d-u}$ and $P_{u,e-d+u}$.
\item
\label{ddiii}
Fix $(\lll B\rr,\lll C\rr)\in W_{u,e-d+u}^v\setminus W_{u,e-d+u}^{v-1}$ and
$\xi\in T_{(\lll A\rr,\lll B\rr,\lll C\rr)}Q_{d,e}^u$ with $d\mu^u_{d,e}(\xi)$  tangent to $W_{d,e}^v\setminus W_{d,e}^{v-1}$. Then $\xi$ is tangent to ${P_{d-u}\times W_{u, e-d+u}^v}$.
\end{enumerate}
\end{prop}

Making use of the rational map ${P_d\times P_e\dashrightarrow P_{d,e}}$ given by $(\lll F\rr,\lll G\rr)\mapsto [F,G]$, it is straightforward to
deduce the properties of $\mu^u_{d,e}$ that appear in Proposition \ref{immdef2} from the corresponding results for~$\wmu^u_{d,e}$ stated in Proposition \ref{immdef1} below.

\begin{prop}
\label{immdef1}
Fix $0\leq v\leq u<d<e$. Set $\wQ_{d,e}^u:=P_{d-u}\times P_{u}\times P_{e-d+u}$.
\begin{enumerate}[label=(\roman*)] 
\item
\label{di}
The multiplication map $\wmu^u_{d,e}:\wQ_{d,e}^u\to P_{d,e}$ is immersive on the open subset
$\wOmega_{d,e}^u:=\{(\lll A\rr,\lll B\rr,\lll C\rr)\in \wQ_{d,e}^u\mid\gcd(A,B,C)=1\}$.
\item
\label{dii}
The normal bundle to $(\wmu^u_{d,e})|_{\wOmega_{d,e}^u}$ is isomorphic to
$$(\tp_1^*\mathcal{E}_{d-u,e+u}(1)\otimes \tp_2^*\cO_{P_{u}\times P_{e-d+u}}(1,1))|_{\tilde{\Omega}_{d,e}^u},$$ where $\tp_1$ and $\tp_2$ are the projections 
of $\wQ_{d,e}^u$ onto $P_{d-u}$ and $P_{u}\times P_{e-d+u}$.
\item
\label{diii}
Fix $(\lll B\rr,\lll C\rr)\in \wW_{u,e-d+u}^v\setminus \wW_{u,e-d+u}^{v-1}$ and
$\xi\in T_{(\lll A\rr,\lll B\rr,\lll C\rr)}\wQ_{d,e}^u$ with $d\wmu^u_{d,e}(\xi)$  tangent to $\wW_{d,e}^v\setminus \wW_{d,e}^{v-1}$. Then $\xi$ is tangent to ${P_{d-u}\times \wW_{u, e-d+u}^v}$.
\end{enumerate}
\end{prop}

\begin{proof}
For $l\geq 0$, we identify $T_{\lll F\rr}P_l$ with $V_l/\lll F\rr$.
Using this identification, one computes that 
$(d\wmu^u_{d,e})_{(\lll A\rr ,\lll B\rr ,\lll C\rr)}$ is given by
\begin{equation}
\label{calculdiff}
\begin{alignedat}{2}
V_{d-u}/\lll A\rr\oplus & V_u/\lll B\rr\oplus V_{e-d+u}/\lll C\rr 
& \hspace{.1em}\to\hspace{.1em}  &  V_{d}/\lll AB\rr\oplus V_{e}/\lll AC\rr  \\
&(H,I,J)&\hspace{.1em}\mapsto\hspace{.1em}  & (BH+AI, CH+AJ)
\end{alignedat}
\end{equation}
for all $(\lll A\rr ,\lll B\rr ,\lll C\rr)\in P_{d-u}\times P_u \times P_{e+d-u}$.
 
If $(H,I,J)\in \Ker((d\wmu^u_{d,e})_{(\lll A\rr ,\lll B\rr ,\lll C\rr)})$, then \eqref{calculdiff} shows that $AB$ divides
$BH+AI$ and $AC$ divides $CH+AJ$, hence that~$A$ divides~$BH$ and~$CH$ (and clearly also~$AH$).
When $(\lll A\rr ,\lll B\rr ,\lll C\rr)\in \wOmega_{d,e}^u$, we deduce that $A$ divides $H$. It then follows that~$B$ divides~$I$ and $C$ divides $J$. This proves \ref{di}.

The Euler exact sequence identifies the tangent bundle $T_{P_d\times P_e}$ with a quotient of $V_{d}\otimes\cO_{P_d\times P_e}(1,0)\oplus V_{e}\otimes\cO_{P_d\times P_e}(0,1)$.
In turn, it identifies $(\wmu^u_{d,e})^*T_{P_d\times P_e}$ with a quotient of 
$V_{d}\otimes\cO_{\wQ_{d,e}^u}(1,1,0)\oplus V_{e}\otimes\cO_{\wQ_{d,e}^u}(1,0,1)$.
For $(\lll A\rr ,\lll B\rr ,\lll C\rr)\in \wQ_{d,e}^u$, consider the linear map
${V_{d}\oplus V_{e}\to V_{e+u}}$ given by
$(F,G)\mapsto BG-CF$. When~$(A,B,C)$ varies, this gives rise to a morphism of sheaves on $\wQ_{d,e}^u$ that reads:
\begin{equation}
\label{sheafipoly}
V_{d}\otimes\cO_{\tq_{d,e}^u}(1,1,0)\oplus V_{e}\otimes\cO_{\tq_{d,e}^u}(1,0,1)\to
V_{e+u}\otimes\cO_{\tq_{d,e}^u}(1,1,1).
\end{equation}
Consider the composition of the morphism \eqref{sheafipoly} and of the quotient morphism $V_{e+u}\otimes\cO_{\tq_{d,e}^u}(1,1,1)\to \tp_1^*\mathcal{E}_{d-u,e+u}\otimes\tp_2^*\cO_{\tq_{d,e}^u}(1,1,1)$. The description of $d\wmu^u_{d,e}$ given in \eqref{calculdiff} shows that, when restricted to $\wOmega_{d,e}^u$, this composition factors through the normal bundle $N_{(\tilde{\mu}^u_{d,e})|_{\tilde{\Omega}_{d,e}^u}}$ of the immersion~$(\wmu^u_{d,e})|_{\tilde{\Omega}_{d,e}^u}$. In this manner, we get a morphism of sheaves $\zeta:N_{(\tilde{\mu}^u_{d,e})|_{\tilde{\Omega}_{d,e}^u}}\to\big(\tp_1^*\mathcal{E}_{d-u,e+u}\otimes\tp_2^*\cO_{\tq_{d,e}^u}(1,1,1)\big)|_{\tilde{\Omega}_{d,e}^u}$. 

To prove \ref{dii}, we show that $\zeta$ is an isomorphism. As it is a morphism between locally free sheaves of the same rank (equal to~${d-u}$), it suffices to check that~$\zeta$ is surjective at the level of fibers. Fix ${(\lll A\rr ,\lll B\rr ,\lll C\rr)\in \wOmega_{d,e}^u}$. By construction, the map $\zeta_{(\lll A\rr ,\lll B\rr ,\lll C\rr)}$ is induced by the linear map ${V_{d}\oplus V_{e}\to V_{e+u}/\lll A\rr}$ given by
\begin{equation}
\label{keyidentification}
(F,G)\mapsto BG-CF.
\end{equation}
It thus suffices to show that \eqref{keyidentification} is surjective, which follows from Lemma \ref{Koszullemma} below.

\begin{lem}
\label{Koszullemma}
Fix $1\leq u<d<e$. If $(A,B,C)\in V_{d-u}\oplus V_u\oplus V_{e-d+u}$ are such that $\gcd(A,B,C)=1$, then $V_{e+u}$ is generated by the multiples of $A$, of $B$ and of $C$.
\end{lem}

\begin{proof}
Set $\cF:=\cO_{\P^1_k}(d-u)\oplus \cO_{\P^1_k}(u)\oplus \cO_{\P^1_k}(e-d+u)$.
The Koszul complex 
\begin{equation}
\label{Koszulcomplex}
0\to \extp^3\cF^{\vee}\to \extp^2\cF^{\vee}\to  \cF^{\vee}\xrightarrow{(A,B,C)} \cO_{\P^1_k}\to 0
\end{equation}
is an exact sequence of sheaves on $\P^1_k$, because $\gcd(A,B,C)=1$. Tensoring \eqref{Koszulcomplex} with $\cO_{\P^1_k}(e+u)$ and taking global sections induces a surjective morphism
\begin{equation}
\label{surjectiveproduct}
H^0(\P^1_k,\cF^{\vee}(e+u))\xrightarrow{(A,B,C)} H^0(\P^1_k,\cO_{\P^1_k}(e+u))
\end{equation}
because $H^1(\P^1_k,\extp^2\cF^{\vee}(e+u))=H^2(\P^1_k,\extp^3\cF^{\vee}(e+u))=0$. The surjectivity of~\eqref{surjectiveproduct} was exactly what we wanted to prove.
\end{proof}

We resume the proof of Proposition \ref{immdef1} and turn to \ref{diii}. Define ${K:=\gcd(B,C)}$ in ${V_{u-v}}$.
Let~$B'\in V_{v}$ and~${C'\in V_{e-d+v}}$ be such that $B=B'K$ and~$C=C'K$.
Write 
$$\xi=(H,I,J)\textrm{ in }V_{d-u}/\lll A\rr\oplus  V_u/\lll B\rr\oplus V_{e-d+u}/\lll C\rr,$$ 
so $d\wmu^u_{d,e}(\xi)=(BH+AI,CH+AJ)$ in $V_{d}/\lll AB\rr\oplus V_{e}/\lll AC\rr$ (see \eqref{calculdiff}). Since $\wmu^v_{d,e}$ is an immersion on $\wOmega_{d,e}^v$ by \ref{di}, the hypothesis that $d\wmu^u_{d,e}(\xi)$ is tangent to $\wW_{d,e}^v\setminus \wW_{d,e}^{v-1}$ implies, by our computation of  $d\wmu^v_{d,e}$ (see \eqref{calculdiff}), that 
\begin{equation}
\label{twodiffs}
(BH+AI,CH+AJ)=(B'H'+AKI',C'H'+AKJ')\textrm{ in }V_{d}/\lll AB\rr\oplus V_{e}/\lll AC\rr
\end{equation}
for some $(H',I',J')\in V_{d-v}/\lll AK\rr\oplus  V_v/\lll B'\rr\oplus V_{e-d+v}/\lll C'\rr$.
It follows from \eqref{twodiffs} that $A$ divides both $B'(KH-H')$ and $C'(KH-K')$, hence that $A$ divides $KH-H'$. Write $KH-H'=AL$ for some $L\in V_{u-v}$. Plugging this into \eqref{twodiffs} shows that 
\begin{equation}
\label{twodiffs2}
(I,J)=(KI'-B'L,KJ'-C'L)\textrm{ in }V_{u}/\lll B\rr\oplus V_{e-d+u}/\lll C\rr.
\end{equation}
By the computation of $d\wmu_{u,e-d+u}^v$ (see \eqref{calculdiff}), equation \eqref{twodiffs2} shows that $\xi$ is tangent to ${P_{d-u}\times \wW_{u, e-d+u}^v}$, as required.
\end{proof}

\subsection{The blow-up tower}

The proof of Theorem \ref{blowup} proceeds by induction, taking advantage of the inductive descriptions of the exceptional divisors~$\whE_{d,e}^u$.
Let us state the precise proposition that we will prove, whose
statement is adapted for an inductive proof, and from which Theorem \ref{blowup} follows easily.
Remark \ref{blowupremarks} (i) below provides an informal explanation of the content of Proposition \ref{blowuprec}.

\begin{prop}
\label{blowuprec}
Fix $1\leq u<d<e$.
\begin{enumerate}[label=(\roman*)] 
\item
\label{ti}
The morphism $\mu_{d,e}^u$ induces a closed immersion 
\begin{equation}
\label{existphi}
\whmu_{d,e}^{\,u}:P_{d-u}\times \whP_{u,e-d+u}\to \beta^{u-1}P_{d,e}
\end{equation}
with normal bundle
$p_1^*\mathcal{E}_{d-u,e+u}(1)\otimes p_2^*\big(\cO_{\whP_{u,e-d+u}}(1,1)(-\hspace{-.1em}\sum_{w=1}^{u-1}\whE_{u,e-d+u}^{w})\big)$,
where~$p_1$ and $p_2$ are the projections of $P_{d-u}\times \whP_{u,e-d+u}$ onto its factors.
\item
\label{tii}
The variety $\beta^{u}P_{d,e}$ is the blow-up of the smooth subvariety
$\beta^{u-1}W_{d,e}^u$ (isomorphic to $P_{d-u}\times \whP_{u,e-d+u}$) in the smooth variety~$\beta^{u-1}P_{d,e}$ with exceptional divisor~$E_{d,e}^u$ is isomorphic to~${P_{d-u, e+u}\times \whP_{u,e-d+u}}$. It is smooth.
\item
\label{tiii}
Fix $(A,B,C,F,G)\in V_{d-u}\oplus V_u\oplus V_{e-d+u}\oplus V_d\oplus V_e$ with $\gcd(B,C)=1$
and $(F,G)$ not in the image of \eqref{calculdiff}.
Then $\lim _{t\to 0}\, [AB+tF,AC+tG]$ equals
$$([A,BG-CF],[B,C])\in P_{d-u, e+u}\times \whP_{u,e-d+u}\stackrel{\ref{tii}}{\simeq} E_{d,e}^u \textrm{ in } \beta^uP_{d,e}.$$
Moreover, if $(A,B,C,F,G)$ are general, then so is $([A,BG-CF],[B,C])$.
\item
\label{tiv}
The $(\beta^u E_{d,e}^w)_{1\leq w\leq u}$ form a strict normal crossings divisor in $\beta^{u}P_{d,e}$.
\item
\label{tv}
 For $u\leq v$, there is a cartesian diagram in which $\lambda$ denotes multiplication:
\begin{equation}
\label{diagmultmap}
\begin{gathered}
\xymatrix@R=1.3em{
\mathcal{X}:=P_{d-v}\times \beta^{u-1}P_{v,e-d+v}\ar[r]^{\hspace{1.3em}\beta^{u-1}\mu_{d,e}^v}
&   \beta^{u-1}P_{d,e}=:\mathcal{Z}    \\
\mathcal{W}:=P_{d-v}\times P_{v-u}\times \whP_{u,e-d+u}\ar[r]^{\hspace{1.3em}(\lambda,\Id)}\ar[u]_{(\Id,\whmu^{\; u}_{v, e-d+v})} 
&P_{d-u}\times \whP_{u,e-d+u}=:\mathcal{Y}. \ar[u]_{\whmu^{\; u}_{d,e}}
}
\end{gathered}
\end{equation}
For $z\in\whmu_{d,e}^{\,u}(\cY)\subset\cZ$, the scheme $(\beta^{u-1}\mu_{d,e}^v)^{-1}(z)$ is finite of degree~$\binom{d-u}{d-v}$.
\item
\label{tvi}
 For $1\leq u< v$, there is a commutative diagram
\begin{equation}
\begin{gathered}
\label{diagbupcomm}
\xymatrix@R=1.3em@C=7em{
P_{d-v}\times \beta^{u-1}P_{v,e-d+v}\ar[r]^{\hspace{2.7em}\beta^{u-1}\mu_{d,e}^v}
&   \beta^{u-1}P_{d,e}    \\
P_{d-v}\times \beta^{u}P_{v,e-d+v}\ar[r]^{\hspace{2.7em}\beta^u\mu_{d,e}^v}\ar[u]_{(\Id,\alpha_{v,e-d+v}^u)}
&   \beta^{u}P_{d,e}. \ar[u]_{\alpha_{d,e}^u}
}
\end{gathered}
\end{equation}
\item
\label{tvii}
 For $1\leq w\leq u< v$, there is a cartesian diagram
 \begin{equation}
\begin{gathered}
\label{diagdivexc}
\xymatrix@R=1.3em @C=1.9em{
P_{d-v}\times \beta^{u}P_{v,e-d+v}\ar[r]^{\hspace{2.7em}\beta^u\mu_{d,e}^v}
&   \beta^{u}P_{d,e}  \\
P_{d-v}\times \beta^uE_{v,e-d+v}^w\ar[r]\ar[u] &\beta^uE_{d,e}^w\ar[u]\\
P_{d-v}\times \beta^{u-w}P_{v-w,e-d+v+w}\times \whP_{w,e-d+w}
\ar[r]^{}\ar@{=}[u]
& \beta^{u-w}P_{d-w,e+w}\times \whP_{w,e-d+w}\ar@{=}[u] 
}
\end{gathered}
\end{equation}
whose lower horizontal arrow is $(\beta^{u-w}\mu_{d-w,e+w}^{v-w},\Id)$ and whose vertical arrows are the natural inclusions.
\end{enumerate}
\end{prop}

\begin{proof}
We prove the proposition by induction on $d$ and, when $d$ is fixed, by induction on $u$.  When we use one of the assertions \ref{ti} to \ref{tvii}, it is always thanks to the induction hypothesis (or because we proved it before). We use without comment that previously studied blow-ups are blow-ups with smooth centers in smooth varieties, hence with smooth total space (by~\ref{tii}).

The existence of $\whmu_{d,e}^{\,u}$ as in (\ref{existphi}) is given by \eqref{diagbupcomm}.
Moreover, diagram~\eqref{diagdivexc} shows that for $1\leq w< u$, the restriction of $\whmu_{d,e}^{\,u}$ over
$\beta^{u-1}E^w_{d,e}$ can be identified with $(\whmu_{d-w,e+w}^{\,u-w},\Id)$ which we know to be a closed immersion by~\ref{ti}. As $\whmu_{d,e}^{\,u}$ is clearly injective over the complement of these loci, this shows that~$\whmu_{d,e}^{\,u}$ is injective.

We now prove that $\whmu_{d,e}^{\,u}$ is immersive. By Proposition \ref{immdef2} \ref{ddi}, the morphism~$\whmu_{d,e}^{\,u}$ is immersive over the complement of the divisors $(\beta^{u-1}E^w_{d,e})_{1\leq w< u}$. It remains to show that $\whmu_{d,e}^{\,u}$  is immersive at a point~$x$ such that $\whmu_{d,e}^{\,u}(x)\in\beta^{u-1}E^w_{d,e}$ for some~${1\leq w< u}$. Fix $\xi\in T_x(P_{d-u}\times \whP_{u,e-d+u})$ with $d\whmu_{d,e}^{\,u}(\xi)=0$. Then~${\xi\in T_x((\whmu_{d,e}^{\,u})^{-1}(\beta^{u-1}E^w_{d,e}))}$, where the tangent space is meant in the schematic sense. By \eqref{diagdivexc}, the restriction
$$\whmu_{d,e}^{\,u}|_{(\whmu_{d,e}^{\; u})^{-1}(\beta^{u-1}E^w_{d,e})}:(\whmu_{d,e}^{\,u})^{-1}(\beta^{u-1}E^w_{d,e})\to \beta^{u-1}E^w_{d,e}$$
can be identified with 
$$P_{d-u}\times \whP_{u-w,e-d+u+w}\times \whP_{w,e-d+w}
\xrightarrow{(\whmu^{\; u-w}_{d-w,e+w},\Id)}\beta^{u-w-1}P_{d-w,e+w}\times \whP_{w,e-d+w},$$
which is immersive by \ref{ti}. This shows that $\xi=0$ and concludes.

Consider the normal bundle of $\whmu_{d,e}^{\,u}$. Proposition~\ref{immdef2}~\ref{ddii} and the behaviour of normal bundles under smooth blow-ups (see \eg \cite[B.6.10]{Fulton}) show that
its restriction to $(\Id,\beta^{u-1})^{-1}(\Omega_{d,e}^u)\subset P_{d-u}\times \whP_{u,e-d+u}$, where~$\Omega_{d,e}^u$ is the open subset of $P_{d-u}\times P_{u,e-d+u}$ defined in Proposition~\ref{immdef2} \ref{ddi}, is isomorphic to 
$$\Big(p_1^*\mathcal{E}_{d-u,e+u}(1)\otimes p_2^*\big(\cO_{\whP_{u,e-d+u}}(1,1)(-\sum_{w=1}^{u-1}\whE_{u,e-d+u}^{w})\big)\Big)|_{\Omega_{d,e}^u}.$$
Since the complement of $(\Id,\beta^{u-1})^{-1}(\Omega_{d,e}^u)$ in $P_{d-u}\times \whP_{u,e-d+u}$ has codimension~$\geq 2$ and since $P_{d-u}\times \whP_{u,e-d+u}$ is a smooth variety, this isomorphism over $\Omega_{d,e}^u$ extends to an isomorphism on all of $P_{d-u}\times \whP_{u,e-d+u}$. This ends the proof of \ref{ti}.

The variety $\beta^u P_{d,e}$ is the blow-up of $\beta^{u-1}W_{d,e}^u$ (isomorphic to $P_{d-u}\times \whP_{u,e-d+u}$ by~\ref{ti}, hence smooth) in the smooth variety $\beta^{u-1}P_{d,e}$. 
The normal bundle computation in \ref{ti} yields the required description of the exceptional divisor, proving ~\ref{tii}. 
Remembering how exactly this normal bundle was identified, in restriction to an appropriate open subset of $\beta^{u-1}P_{d,e}$, with $p_1^*\mathcal{E}_{d-u,e+u}(1)\otimes p_2^*(\cO_{\whP_{u, e-d+u}}(1,1))$ (see the proof of Proposition \ref{immdef1} \ref{dii} and especially \eqref{keyidentification}) implies assertion \ref{tiii}.

Using \ref{tiv} by induction shows that the $(\beta^{u-1} E_{d,e}^w)_{1\leq w< u}$
form a strict normal crossings divisor in $\beta^{u-1}P_{d,e}$. In addition, diagram \eqref{diagdivexc} shows that the inverse images of these divisors in $P_{d-u}\times \whP_{u, e-d+u}$ by the closed immersion \eqref{existphi} are the~$(P_{d-u} \times\whE_{u,e-d+u}^{w})_{1\leq w< u}$, which form a strict normal crossings divisor in~$P_{d-u}\times \whP_{u, e-d+u}$ (apply \ref{tiv} again by induction). This implies that the divisors~$(\beta^{u} E_{d,e}^w)_{1\leq w\leq u}$ form a strict normal crossings divisor and proves \ref{tiv}.

All the maps in \eqref{diagmultmap} are well-defined, by~\ref{tvi}. Diagram \eqref{diagmultmap} is commutative as its restriction to the dense open subset $P_{d-v}\times P_{v-u}\times U_{u,e-d+u}$ of~$\cW$ commutes. Consider the fiber product $\mathcal{V}:=\mathcal{X}\times_{\mathcal{Z}}\mathcal{Y}$. By \ref{ti}, the morphisms~${\whmu_{d,e}^{\,u}:\cY\to\cZ}$ and~${(\Id, \whmu_{v,e-d+v}^{\,u}):\cW\to\cX}$ are closed immersions. As a consequence, so is the morphism~${\kappa:\mathcal{W}\to\mathcal{V}}$ induced by~\eqref{diagmultmap}. Let $\cZ^0\subset \cZ$ be the complement of the divisors $(\beta^{u-1}E^w_{d,e})_{1\leq w<u}$. Let~$\cX^0$, $\cY^0$, $\cW^0$ and $\cV^0$ be the inverse images of~$\cZ^0$ in~$\cX$,~$\cY$,~$\cW$ and $\cV$ respectively. The concrete descriptions of these spaces show that the closed immersion $\kappa|_{\cW^0}:\cW^0\to\cV^0$ is bijective. By Proposition~\ref{immdef2}~\ref{ddiii}, it also induces bijections at the level of tangent spaces. As~$\cW^0$ is smooth, this implies that $\kappa|_{\cW^0}$ is an isomorphism. 

We now prove the last assertion of \ref{tv}. Fix $z\in\whmu_{d,e}^{\,u}(\cY)\subset\cZ$. As~${(\lambda,\Id):\cW\to\cY}$ is a finite surjective morphism between smooth varieties, it is flat by \cite[Theorem~23.1]{Matsumura}. Its degree is $\binom{d-u}{d-v}$, as the degree of a general fiber of $(\lambda,\Id)$ is the number of degree~$d-v$ factors (up to a scalar) of a general degree~$d-u$ polynomial.
 As~$\kappa|_{\cW^0}$ is an isomorphism, this implies the last assertion of~\ref{tv} when~$z\in \cZ^0$. To conclude, it remains to deal with the case where $z\in \beta^{u-1}E^w_{d,e}\subset\cZ$ for some~$1\leq w<u$. By~\ref{tvii}, the restriction of~$\beta^{u-1}\mu_{d,e}^v:\cX\to\cZ$ over~$\beta^{u-1}E^w_{d,e}$ can be identified with $(\beta^{u-w-1}\mu^{v-w}_{d-w,e+w},\Id)$. Since $z\in\whmu^{\,u}_{d,e}(\cY)$, another application of \ref{tvii} shows that $z$ belongs to the subset of~$\beta^{u-w-1}P_{d-w,e+w}\times \whP_{w,e-d+w}$ that is the image of~$(\whmu^{\,u-w}_{d-w,e+w}, \Id)$. The inverse image of $z$ by $(\beta^{u-w-1}\mu^{v-w}_{d-w,e+w},\Id)$ therefore has degree~$\binom{(d-w)-(u-w)}{(d-w)-(v-w)}=\binom{d-u}{d-v}$ by induction.

To complete the proof of \ref{tv}, we show that \eqref{diagmultmap} is cartesian, \ie that the closed immersion~$\kappa:\cW\to \cV$ is an isomorphism. Let $\nu:\cV\to\cY$ be the finite morphism induced by \eqref{diagmultmap}. Consider the surjective morphism of sheaves
\begin{equation}
\label{surjectionpf}
\nu_*\cO_{\cV}\to\nu_*\cO_{\cW}
\end{equation}
obtained by pushing forward by $\nu$ the surjection $\cO_{\cV}\to\cO_{\cW}$. The sheaf $\nu_*\cO_{\cW}$ is locally free of rank~$\binom{d-u}{d-v}$ because ${(\lambda,\Id):\cW\to\cY}$ is finite flat of degree ~$\binom{d-u}{d-v}$ (as we saw above). The sheaf~$\nu_*\cO_{\cV}$ is also locally free of rank $\binom{d-u}{d-v}$ because all its fibers have rank $\binom{d-u}{d-v}$ by the last assertion of \ref{tv}, and because $\cY$ is reduced. It now follows from Nakayama's lemma that the surjection \eqref{surjectionpf} is an isomorphism, and we deduce that $\kappa:\cW\to\cV$ is an isomorphism. 

The commutative diagram \eqref{diagbupcomm} is constructed from diagram \eqref{diagmultmap} by blowing up the vertical closed immersions. This proves \ref{tvi}. 

The commutative diagram \eqref{diagdivexc} for $w=u$ is the restriction of \eqref{diagbupcomm} to the exceptional divisors of the blow-ups. It is cartesian because so is \eqref{diagmultmap}. The description of the exceptional divisors follows from \ref{tii}. To show that the map induced by $\beta^u\mu_{d,e}^v$ between the exceptional divisors is $(\mu_{d-u,e+u}^{v-u},\Id)$, we compute it at a general point $(\lll H\rr, [A,BG-CF],[B,C])$ of $P_{d-v}\times P_{v-u,e-d+v+u}\times\whP_{u,e-d+u}$ (see~\ref{tiii}). By assertion \ref{tiii}, the image of this point by the map induced by $\beta^u\mu_{d,e}^v$~is 
$$\lim_{t\to 0}\mu_{d,e}^v(\lll H\rr, [AB+tF, AC+tG])=\lim_{t\to 0}\,[ABH+tFH, ACH+tGH],$$
which is identified, by means of \ref{tiii}, to 
$$([AH, BGH-CFH],[B,C])=(\mu_{d-u,e+u}^{v-u},\Id)(\lll H\rr, [A,BG-CF],[B,C]).$$

It remains to prove \ref{tvii} when $w<u$. Apply diagram \eqref{diagdivexc} to~$(u-1,v,w)$ instead of~$(u,v,w)$. The restrictions to the four spaces of this diagram of the image of the closed immersion $\whmu_{d,e}^{\,u}:P_{d-u}\times \beta^{u-1}P_{u, e-d+u}\to \beta^{u-1}P_{d,e}$ can be computed using \ref{tv} and \ref{tvii}. Blowing up these loci yields the desired commutative diagram~\eqref{diagdivexc}. It is cartesian because the inverse image in ${P_{d-v}\times \beta^uP_{v,e-d+v}}$ of the Cartier divisor $\beta^uE_{d,e}^w\subset\beta^uP_{d,e}$ is a Cartier divisor which equals~${P_{d-v}\times \beta^uE^w_{v, e-d+v}}$ away from a codimension $\geq 2$ subset, and hence which coincides with it
\end{proof}

\begin{rems}
\label{blowupremarks}
(i) Let us spell out what Proposition \ref{blowuprec} tells us about the blow-up tower $\whP_{d,e}\to P_{d,e}$. Fix $1\leq u<d$. During the first $u-1$ blow-ups, the source of the multiplication morphism $\mu_{d,e}^u:P_{d-u}\times P_{u,e-d+u}\to P_{d,e}$ undergoes the very same blow-up process, on the second factor (that is $P_{d-u}\times \whP_{u,e-d+u}\to P_{d-u}\times P_{u,e-d+u}$). This has the effect of turning this multiplication map into an embedding, hence to resolve the singularities of its image which becomes isomorphic to $P_{d-u}\times \whP_{u,e-d+u}$. This image is then blown up, and the exceptional divisor $E_{d,e}^u$ identifies with the projective bundle $P_{d-u, e+u}\times \whP_{u,e-d+u}$ over it. During the last $d-u-1$ blow-ups, this exceptional divisor undergoes the very same blow-up process, on the first factor (that is $\whP_{d-u, e+u}\times \whP_{u,e-d+u}\to P_{d-u, e+u}\times \whP_{u,e-d+u}$). The resulting divisor $\whE_{d,e}^u$ is therefore isomorphic to $\whP_{d-u, e+u}\times \whP_{u,e-d+u}$.

(ii) When $d\geq 2$, it follows from Proposition \ref{blowuprec} \ref{tii} that the last blow\nobreakdash-up $\alpha_{d,e}^{d-1}:\whP_{d,e}=\beta^{d-1}P_{d,e}\to\beta^{d-2}P_{d,e}$ is the blow-up of a smooth
divisor, hence an isomorphism. It was therefore not necessary to blow up the strict transform of~$W_{d,e}^{d-1}$ to construct~$\whP_{d,e}$. However, it is important to take into account and study its exceptional divisor $\whE_{d,e}^{\,d-1}\subset \whP_{d,e}$, that is the strict transform of the discriminant.

(iii) The inductive nature of Proposition \ref{blowuprec} and its proof makes it necessary to take into account, on an equal footing, the slightly degenerate $d=1$ case.
\end{rems}

\subsection{Proof of Theorem \ref{blowup}}

Assertions \ref{bi} and \ref{bii} are particular cases of Proposition \ref{blowuprec}~\ref{tii} and \ref{tiv}. 

To prove \ref{biii}, \ref{biv} and \ref{bv}, we use implicitly that the strict transforms and the pullbacks of the divisors that we consider under the various blow-ups that appear always coincide, because they coincide away from a subset of codimension $\geq 2$. 

For~$v<u$, the restriction $(\beta^{u-1}E_{d,e}^v)|_{\beta^{u-1}W_{d,e}^u}$ is isomorphic to $q_2^*\whE^v_{u,e-d+u}$ (where $q_2:\beta^{u-1}W_{d,e}^u\to \whP^v_{u,e-d+u}$ is the second projection induced by the isomorphism \eqref{existphi}), because of \eqref{diagdivexc}. Pulling back this isomorphism to $\whE_{d,e}^u$ proves~\ref{biii}.

For $v>u$, diagram \eqref{diagdivexc} shows that the restriction of the subscheme $\beta^{v-1}W_{d,e}^v$ of $\beta^{v-1}P_{d,e}$ to $\beta^{v-1}E_{d,e}^u$  is isomorphic to $(q_1)^{-1}(\beta^{v-u-1}W_{d-u,e+u}^{v-u})$ (where we let ${q_1:\beta^{v-1}E_{d,e}^u\to \beta^{v-u-1}P_{d-u,e+u}}$ be the first projection induced by the bottom right equality in \eqref{diagdivexc}). Pulling back this isomorphism to $\whE_{d,e}^u$ proves~\ref{biv}.

The exceptional divisor $E_{d,e}^u\subset \beta^{u}P_{d,e}$ is the projectivization of the normal bundle of $\beta^{u-1}W_{d,e}^u$ in $\beta^{u-1}P_{d,e}$ which, by Proposition \ref{blowuprec} \ref{ti}, is isomorphic to 
\begin{equation}
\label{normaltwist}
q_1^*\mathcal{E}_{d-u,e+u}\otimes\Big(q_1^*\cO_{P_{d-u}}(1)\otimes q_2^*\big(\cO_{\whP_{u,e-d+u}}(1,1)(-\sum_{w=1}^{u-1}\whE_{u,e-d+u}^{w})\big)\Big)
\end{equation}
(where $q_1:\beta^{u-1}W_{d,e}^u\to P_{d-u}$ and $q_2:\beta^{u-1}W_{d,e}^u\to\whP_{u,e-d+u}$ are the two projections in~\eqref{existphi}). The line bundle $\cO_{\beta^uP_{d,e}}(E_{d,e}^u)|_{E_{d,e}^u}$ is the antiample tautological line bundle of this projective bundle. As the antiample tautological line bundle of the (isomorphic) projectivization of $q_1^*\mathcal{E}_{d-u,e+u}$ is by definition the line bundle~$(q'_1)^*\cO_{P_{d-u,e+u}}(0,-1)$, we deduce from~\eqref{normaltwist} and \cite[II, Lemma 7.9]{Hartshorne} that 
 $$\cO_{\beta^uP_{d,e}}(E_{d,e}^u)|_{E_{d,e}^u}\simeq (q'_1)^*\cO_{P_{d-u,e+u}}(1,-1)\otimes (q'_2)^*\cO_{\whP_{u,e-d+u}}(1,1)(-\sum_{w=1}^{u-1}\whE_{u,e-d+u}^{w})$$
 (where $q'_1:E_{d,e}^u\to P_{d-u,e+u}$ and $q'_2:E_{d,e}^u\to \whP_{u, e-d+u}$ are the projections induced by Proposition \ref{blowuprec} \ref{tii}).
 Pulling back this isomorphism to $\whE_{d,e}^u$ proves~\ref{bv}.

Finally, the definition of $\mu_{d,e}^u:P_{d-u}\times P_{u,e-d+u}\to P_{d,e}$ given in \eqref{defphi} implies that~$(\mu_{d,e}^u)^*\cO_{P_{d,e}}(l,m)\simeq q_1^*\cO_{P_{d-u}}(l+m)\otimes q_2^*\cO_{P_{u, e-d+u}}(l,m)$ (where $q_1$ and $q_2$ are the projections of $P_{d-u}\times P_{u,e-d+u}$ onto its factors). Pulling back this isomorphism to $\whE_{d,e}^u$ computes $\cO_{\whP_{d,e}}(l,m)|_{\whE_{d,e}^u}$ and proves \ref{bvi}.
\qed

\section{The multigraded Hilbert scheme}
\label{secmultigraded}

In this section, we give an interpretation of $\whP_{d,e}$ as a mutigraded Hilbert scheme in the sense of \cite{multigraded} (see Theorem \ref{thHilb}).

\subsection{The variety \texorpdfstring{$\Hilb_{d,e}$}{Hilbde}}
\label{parmultigraded}

Let $\mathbb{G}_m$ act on $\A^2_k$ by homotheties.
Let $Z\subset\A^2_k$ be a~$\mathbb{G}_m$\nobreakdash-invariant closed subscheme. 
The Hilbert function $l\mapsto \HF_Z(l)$ of $Z$ associates with~$l\in\Z$ the dimension of the subspace of $V_l$ consisting of equations of~$Z$. Our convention differs from~\cite{multigraded}, which considers the dimension of the quotient (it is more convenient for us to manipulate equations of~$Z$ rather than functions on~$Z$).

For $1\leq d<e$, we consider the function $l\mapsto \HF_{d,e}(l)$ defined by:
\begin{equation}
\label{defHF}
\HF_{d,e}(l)=\left\{
    \begin{array}{llll}
       0&\textrm{ if } \hspace{0.3em} l\leq d-1, \\
       l+1-d&\textrm{ if } \hspace{0.3em}d\leq l\leq e-1, \\
       2l+2-d-e&\textrm{ if } \hspace{0.3em}e\leq l\leq d+e-1, \\
       l+1&\textrm{ if } \hspace{0.3em}d+e\leq l. 
    \end{array}
\right.
\end{equation}

\begin{lem}
\label{HF}
Fix $1\leq d<e$. Choose $[F,G]\in U_{d,e}$. Set $Z:=\{F=G=0\}\subset\A^2_k$. 
Then $\HF_Z(l)=\HF_{d,e}(l)$ for all $l\in\Z$.
\end{lem}

\begin{proof}
As $[F,G]\in U_{d,e}$, one can consider the Koszul exact sequence associated with the regular sequence $(F,G)$ on $\A^2_k$. It reads:
\begin{equation}
\label{Koszulagain}
0\to \cO_{\A^2_k}\xrightarrow{(G,-F)}\cO_{\A^2_k}^{\oplus 2}\xrightarrow{(F,G)}\cO_{\A^2_k}\to\cO_Z\to  0.
\end{equation}
Taking global sections in \eqref{Koszulagain} and retaining the terms of appropriate homogeneous degrees shows that $\HF_Z(l)=\dim_k(V_{l-d})+\dim_k(V_{l-e})-\dim_k(V_{l-d-e})$ for all $l\in\Z$. Computing this function shows that it coincides with $\HF_{d,e}(l)$.
\end{proof}

Let $\Hilb_{d,e}$ be the multigraded Hilbert scheme parameterizing $\mathbb{G}_m$-invariant subschemes 
of~$\A_k^2$ with Hilbert function $\HF_{d,e}$, as defined and constructed in \cite[Theorem 1.1]{multigraded}. We use the same notation for a point of $\Hilb_{d,e}$ and for the associated subscheme~of~$\A^2_k$.
Let $\sigma_{d,e}: U_{d,e}\to \Hilb_{d,e}$ be the morphism associating with~$[F,G]\in U_{d,e}$ the subscheme $\{F=G=0\}\subset\A^2_k$ (see Lemma \ref{HF}).

\begin{prop}
\label{hilbhilb}
Fix $1\leq d<e$. The morphism $\sigma_{d,e}$ realizes $\Hilb_{d,e}$ as a smooth projective compactification of~$U_{d,e}$. There exists a morphism~${\pi_{d,e}: \Hilb_{d,e}\to P_{d,e}}$ that restricts to the identity on $U_{d,e}$.
\end{prop}

\begin{proof}
The variety $\Hilb_{d,e}$ is projective by \cite[Corollary 1.2]{multigraded} and smooth and connected by
\cite[Theorem 1]{Evain} or \cite[Theorem~1.1]{SmithMaclagan} (these facts appear in \cite[Theorems 2.9 and 3.13]{Iarrobino} with the caveat that Iarrobino ignores nilpotents).

For $Z\in \Hilb_{d,e}$, it follows from \eqref{defHF} that $Z$ satisfies a unique nonzero equation~$F\in V_d$ (up to multiplication by scalar), and a unique nonzero equation~$G\in V_e$ (up to a multiple of~$F$ and multiplication by a scalar). One can therefore define~$\pi_{d,e}: \Hilb_{d,e}\to P_{d,e}$ by the formula $\pi_{d,e}(Z)=[F,G]$.

If $\pi_{d,e}(Z)=[F,G]$ is in $U_{d,e}$, then $Z=\{F=G=0\}$ (because $Z\subset\{F=G=0\}$ and $\HF_Z=\HF_{\{F=G=0\}}$ by Lemma \ref{HF}). The section $\sigma_{d,e}$ of $\pi_{d,e}$ above~$U_{d,e}$ therefore shows that $\pi_{d,e}$ is an isomorphism above~$U_{d,e}$. It follows that $\pi_{d,e}$ is birational and that $\Hilb_{d,e}$ is a smooth compactification of~$U_{d,e}$. 
\end{proof}

\begin{rem}
\label{remHilb}
The subschemes of~$\A^2_k$ parameterized by $\Hilb_{d,e}$ satisfy all the equations of degree $\geq d+e$ by \eqref{defHF}, so they are supported on the origin. We may thus view them as finite subschemes of $\P^2_k$. The induced morphism~${\Hilb_{d,e}\to\Hilb(\P^2_k)}$ is a monomorphism, as one sees from the description of the functors of points of these two schemes, hence a closed immersion by \cite[Corollaire 18.12.6]{EGA44}.
\end{rem}

\subsection{Boundary divisors in \texorpdfstring{$\Hilb_{d,e}$}{Hilbde}}
\label{parboundary}

In this paragraph, we construct and study divisors in $\Hilb_{d,e}$ that will turn out to correspond to the divisors $\whE_{d,e}^u$ in $\whP_{d,e}$.

\begin{lem}
\label{excexc1}
Fix $1\leq u<d<e$.
 There exists an injective morphism $$\psi_{d,e}^u:\Hilb_{d-u,e+u}\times\Hilb_{u,e-d+u}\to\Hilb_{d,e}$$ such that for $Z\in\Hilb_{d-u,e+u}$ and $Z'\in\Hilb_{u,e-d+u}$ with $\pi_{d-u,e+u}(Z)=[F,G]$, the subscheme $\psi_{d,e}^u(Z,Z')$ is defined in $\A^2_k$ by the equations of the form $FK$ for any equation $K$ of $Z'$, and by the equations of degree $\geq e+u$ of $Z$.
\end{lem}

\begin{proof}
Let us first show that the morphism $\psi_{d,e}^u$ is well-defined. Fix $Z\in\Hilb_{d-u,e+u}$ and ${Z'\in\Hilb_{u,e-d+u}}$ with $\pi_{d-u,e+u}(Z)=[F,G]$.

Let $Y\subset \A^2_k$ be the subscheme defined by the equations of the form $FK$ for any equation $K$ of $Z'$. It is easy to describe $\HF_Y$ from $\HF_{Z'}$ (that is known by Lemma~\ref{HF}).
Since $\psi_{d,e}^u(Z,Z')$ is defined in $Y$ by additional equations of degrees~${\geq e+u}$, it follows that $\HF_{\psi_{d,e}^u(Z,Z')}(l)$ coincides with $\HF_{d,e}(l)$
for $l<e+u$. 

Moreover, again by Lemma \ref{HF}, the equations of degrees $\geq e+u$ of $Y$ are exactly the multiples of $F$, hence are also equations of~$Z$. It follows that the equations of degrees $\geq e+u$ of
$Z$ and $\psi_{d,e}^u(Z,Z')$ are the same. As $\HF_Z$ is known by Lemma~\ref{HF}, one checks that 
$\HF_{\psi_{d,e}^u(Z,Z')}(l)$ coincides with $\HF_{d,e}(l)$
for $l\geq e+u$ as well. We have proven that
$\psi_{d,e}^u(Z,Z')\in\Hilb_{d,e}$, hence that $\psi_{d,e}^u(Z,Z')$ is well-defined.

To prove that $\psi_{d,e}^u$ is injective, notice that $F$ is recovered from $\psi_{d,e}^u(Z,Z')$ as the
greatest common divisor of the equations of $\psi_{d,e}^u(Z,Z')$ of degrees $<e+u$, that the equations of $Z'$ are recovered by dividing these equations by $F$, and that the additional equations of $Z$ can also be recovered since they are the equations of~$\psi_{d,e}^u(Z,Z')$ of degrees $\geq e+u$.
\end{proof}

We will show later, in Corollary \ref{coropsi}, that $\psi_{d,e}^u$ is a closed immersion.

For $1\leq u<d<e$, define $D_{d,e}^u\subset \Hilb_{d,e}$ to be the image of $\psi_{d,e}^u$. As $\psi_{d,e}^u$ is injective, a dimension computation shows that $D_{d,e}^u$ is a divisor in $ \Hilb_{d,e}$.

\begin{lem}
\label{excexc2}
Fix $1\leq d<e$.
The irreducible components of the complement of~$U_{d,e}$ in $\Hilb_{d,e}$ are exactly the divisors $(D_{d,e}^u)_{1\leq u<d}$.
\end{lem}

\begin{proof}
One verifies that $\pi_{d,e}(\Ima(\psi_{d,e}^u))=W_{d,e}^u$. It follows that the $(D_{d,e}^u)_{1\leq u<d}$ are distinct and do not meet $U_{d,e}$. Let us show conversely that if $Y\in \Hilb_{d,e}\setminus U_{d,e}$, then~$Y$ is included in one of the $(D_{d,e}^u)_{1\leq u<d}$. 

Let $1\leq u<d$ be such that $\pi_{d,e}(Y)\in W_{d,e}^u\setminus W_{d,e}^{u-1}$. Write $\pi_{d,e}(Y)=[AB,AC]$ with $A$ of degree $d-u$. Set $Z':=\{B=C=0\}$. Let $Z$ be the subscheme defined by the vanishing of $A$ and of all the equations of $Y$ of degrees~$\geq e+u$. One checks that $Z\in \Hilb_{d-u,e+u}$, that $Z'\in\Hilb_{u,e-d+u}$ and that~$Y=\psi_{d,e}^u(Z,Z')$, so $Y\in D^u_{d,e}$.
\end{proof}

Let $\rho_{d,e}:\whP_{d,e}\dashrightarrow \Hilb_{d,e}$ be the birational map inducing the identity on $U_{d,e}$.

\begin{lem}
\label{excexc3}
Fix $1\leq u<d<e$.
The birational map $\rho_{d,e}$ sends the generic point of~$\whE^u_{d,e}$ to the generic point of $D^u_{d,e}$ and induces a commutative diagram:
\begin{equation}
\begin{gathered}
\label{diagbexcdivcomm}
\xymatrix@R=1.3em@C=9em{
\whE_{d,e}^u\ar@{-->}[r]^{(\rho_{d,e})|_{\widehat{E}^{\; u}_{d,e}}
}
&   D_{d,e}^u   \\
\whP_{d-u,e+u}\times\whP_{u,e-d+u}\ar@{-->}[r]^{\hspace{0em}(\rho_{d-u,e+u},\,\rho_{u,e-d+u})\hspace{1.5em}}\ar[u]^{\eqref{canoniso}}_{\varphi_{d,e}^u}
&   \Hilb_{d-u,e+u}\times\Hilb_{u,e-d+u}. \ar[u]_{\psi_{d,e}^u}
}
\end{gathered}
\end{equation}
\end{lem}

\begin{proof}
The rational map $\rho_{d,e}$ is defined on an open subset of $\whP_{d,e}$whose complement has codimension $\geq 2$ in $\whP_{d,e}$ because $\whP_{d,e}$ is smooth and $ \Hilb_{d,e}$
is projective. This locus of definition therefore intersects the divisor $\whE^u_{d,e}$.

Fix~$x=([A,H],[B,C])\in \whP_{d-u,e+u}\times\whP_{u,e-d+u}\simeq \whE_{d,e}^u$ general.  By Proposition~\ref{blowuprec}~\ref{tiii}, we may assume that $H=BG-CF$ for some ${F\in V_{d}}$ and $G\in V_{e}$, and that $\varphi_{d,e}^u(x)=\lim_{t\to 0}\,[AB+tF,AC+tG]$ (where the limit is taken in $\whP_{d,e}$). The element $\rho_{d,e}(\varphi_{d,e}^u(x))$ of $\Hilb_{d,e}$ therefore equals
$$\lim_{t\to 0}\rho_{d,e}([AB+tF,AC+tG]).$$
It follows that the subscheme $\rho_{d,e} (\varphi_{d,e}^u(x))$ of $\A^2_k$ satisfies the equations $AB$ and $AC$, as well as $\lim_{t\to 0} \frac{1}{t}\big(B(AC+tG)-C(AB+tF)\big)=BG-CF=H$. It is therefore included in~$\psi_{d,e}^u(\rho_{d-u,e+u}([A,H]),\rho_{u,e-d+u}([B,C]))$, hence equal to it since it has the same Hilbert function.
\end{proof}

\subsection{Semi-ample line bundles on \texorpdfstring{$\Hilb_{d,e}$}{Hilbde}}
\label{parsemiample}

For $l\in\Z$, let $\cF_{d,e}^l\subset V_l\otimes\cO_{\Hilb_{d,e}}$ be the locally free sheaf of rank $\HF_{d,e}(l)$ on $\Hilb_{d,e}$ whose fiber over $Z\in \Hilb_{d,e}$ is the space of degree~$l$ equations of $Z$.
For $i\in\Z$, define $\cM_{d,e}^i:=\det(\cF_{d,e}^{e+i-1})^{-1}$.

\begin{prop}
\label{Lisemiample}
Fix $1\leq d<e$.
\begin{enumerate}[label=(\roman*)] 
\item
\label{Li}
The line bundle $\cM^i_{d,e}$ on $\Hilb_{d,e}$ is semi-ample for all $i\in\Z$
\item
\label{Lii}
The line bundle $\cM^i_{d,e}$ on $\Hilb_{d,e}$ is trivial for $i\geq d$ or $i\leq d-e$.
\end{enumerate}
\end{prop}

\begin{proof}
For $l\geq 0$, the sheaf $\cF_{d,e}^l$ is a subsheaf of the constant vector bundle with fiber $V_l$ over $\Hilb_{d,e}$, with a locally free quotient. It therefore induces a morphism
\begin{equation}
\label{defchi}
\chi_{d,e}^l:\Hilb_{d,e}^l\to \Grass\big(\HF_{d,e}(l),V_l\big)
\end{equation}
to the Grassmannian of rank $\HF_{d,e}(l)$ subspaces of $V_l$ (this construction originates from \cite{BS}). The pull-back by $\chi_{d,e}^l$ of the (ample) Pl\"ucker line bundle on the Grassmannian is equal to $\det(\cF_{d,e}^l)^{-1}$, which is therefore semi-ample.

The sheaf $\cF^l_{d,e}$ is zero for $l\leq d-1$ and equal to $V_l\otimes \cO_{\Hilb_{d,e}}$ for $l\geq e+d-1$ (see \eqref{defHF}). In both cases, its determinant is trivial.

Applying these facts with $l=e+i-1$ proves \ref{Li} and \ref{Lii}
\end{proof}

To  compute the line bundle $\cM_{d,e}^i$, we define $\cO_{\Hilb_{d,e}}(l,m):=\pi_{d,e}^*\cO_{P_{d,e}}(l,m)$, where ${\pi_{d,e}: \Hilb_{d,e}\to P_{d,e}}$ is the morphism of Proposition \ref{hilbhilb}. 

\begin{prop}
\label{Licalcul}
For $1\leq d<e$ and $0\leq i\leq d$, one has
\begin{equation}
\label{eqLi}
\cM_{d,e}^i\simeq\cO_{\Hilb_{d,e}}(e-d+i,i)\Big(\sum_{u=1}^{i-1}-(i-u)D_{d,e}^u\Big).
\end{equation}
\end{prop}

\begin{rem}
For $i=0$, one must interpret \eqref{eqLi} as $\cM_{d,e}^0\simeq\cO_{\Hilb_{d,e}}(e-d,0)$.
\end{rem}

\begin{proof}
For $l\geq d$, consider the short exact sequence
\begin{equation}
\label{d}
0\to V_{l-d}\otimes\cO_{P_{d,e}}(-1,0)\to V_l\otimes\cO_{P_{d,e}}\to \cQ_{d,e}^l\to0
\end{equation}
of sheaves on $P_{d,e}$ whose fiber over $[F,G]$ is given by
$$0\to V_{l-d}\xrightarrow{\cdot F} V_l\to V_l/\lll F\rr_l\to 0.$$
For $l\geq e-1$, consider the composition
\begin{equation}
\label{e}
V_{l-e}\otimes \cO_{P_{d,e}}(0,-1)\to V_{l-e}\otimes \cQ_{d,e}^e\to \cQ_{d,e}^l
\end{equation}
of the inclusion and of the multiplication map, whose fiber over $[F,G]$ reads
$$V_{l-e}\xrightarrow{\cdot G}V_l/\lll F\rr_l.$$
Consider the pull-back diagram
\begin{equation}
\label{de}
\begin{gathered}
\xymatrix@C=1.5em@R=1.3em{
0\ar[r]&V_{l-d}\otimes\cO_{P_{d,e}}(-1,0)\ar[d]\ar[r]^{}&\cG_{d,e}^l \ar[d]\ar[r]^{}&V_{l-e}\otimes \cO_{P_{d,e}}(0,-1)\ar[d]\ar[r]^{}&0     \\
0\ar[r]&V_{l-d}\otimes\cO_{P_{d,e}}(-1,0)\ar[r]^{}&V_l\otimes\cO_{P_{d,e}}\ar[r]^{}&\cQ_{d,e}^l\ar[r]^{}&0  
}
\end{gathered}
\end{equation}
of locally free sheaves on $P_{d,e}$ whose bottom row is \eqref{d} and whose right vertical arrow is \eqref{e}.
By construction, the image of the fiber at $[F,G]$ of the middle vertical arrow of \eqref{de} is the subspace of $V_l$ spanned by the multiples of $F$ and $G$. It follows that the pull-back 
$\pi_{d,e}^*\cG_{d,e}^l\to V_l\otimes\cO_{\Hilb_{d,e}}$ of the middle vertical arrow of~\eqref{de} factors through $\cF_{d,e}^l$, giving rise to a morphism
\begin{equation}
\label{locfreemap}
\pi_{d,e}^*\cG_{d,e}^l\to\cF_{d,e}^l.
\end{equation}

Now, set $l=e+i-1$. As $0\leq i\leq d$, the source and the target of \eqref{locfreemap} are both locally free sheaves of rank $2i+e-d$ on $\Hilb_{d,e}$. As the fiber of $\cF^l_{d,e}$ over $[F,G]\in U_{d,e}$ is spanned by the multiples of $F$ and $G$, the morphism \eqref{locfreemap} is surjective over $U_{d,e}$, hence an isomorphism over this locus. We deduce that \eqref{locfreemap} is injective. Letting $\cH_{d,e}^l$ denote its cokernel, we get a short exact sequence
\begin{equation}
\label{soscoh}
0\to \pi_{d,e}^*\cG_{d,e}^l\to\cF_{d,e}^l\to\cH_{d,e}^l\to 0.
\end{equation}

By Lemma \ref{excexc2}, the support of $\cH_{d,e}^l$ is set-theoretically included in $\bigcup_{u=1}^{d-1}D_{d,e}^u$. Lemma~\ref{multexc} below shows that, in restriction to an appropriate neighborhood of the generic point of~$D_{d,e}^u$ in~$\Hilb_{d,e}$, the coherent sheaf $\cH^l_{d,e}$ is the pushforward of a vector bundle on~$D_{d,e}^u$, which is of rank~$0$ if~${u>i-1}$ and of rank $i-u$ if~$u\leq i-1$.

Taking the determinant of \eqref{soscoh} and utilizing the first row of \eqref{de} to compute~$\det(\cG_{d,e}^l)$ yields $\cM^i_{d,e}\simeq\cO_{\Hilb_{d,e}}(e-d+i,i)(\sum_{u=1}^{i-1}-(i-u)D_{d,e}^u)$.
\end{proof}

We used the following lemma.

\begin{lem}
\label{multexc}
Fix $1\leq u<d<e$. Let $l$ be an integer such that $e-1\leq l\leq e+d-1$. Let $\cH_{d,e}^l$ be the coherent sheaf on $\Hilb_{d,e}$ defined in the proof of Theorem \ref{thHilb}.
\begin{enumerate}[label=(\roman*)] 
\item
\label{cohi}
If $u>l-e$, then $\cH_{d,e}^l=0$ in a neighborhood of the generic point of~$D_{d,e}^u$.
\item
\label{cohii}
If $u\leq l-e$, then $\cH_{d,e}^l$ is the pushforward of a vector bundle of rank ${l-e-u+1}$ on $D_{d,e}^u$, in some neighborhood of the generic point of~$D_{d,e}^u$.
\end{enumerate}
\end{lem}

\begin{proof}
Choose $(A,B,C,F,G)\in V_{d-u}\oplus V_u\oplus V_{e-d+u}\oplus V_d\oplus V_e$ general. 
The formula $t\mapsto[AB+tF,AC+tG]$ defines a morphism $\Spec(k(\hspace{-.1em}(t)\hspace{-.1em}))\to \Hilb_{d,e}$ which extends to a morphism $\iota:\Spec(k[[t]])\to \Hilb_{d,e}$ by properness of $\Hilb_{d,e}$. By Lemma \ref{excexc3} and Proposition \ref{blowuprec} \ref{tiii}, the image by~$\iota$ of the special point of $\Spec(k[[t]])$ is a general point~$x\in D_{d,e}^u$.

Pulling back \eqref{locfreemap} by $\iota$ and splitting the first row of \eqref{de} after pullback by $\pi_{d,e}\circ\iota$ yields the morphism of $k[[t]]$-modules
\begin{equation}
\label{kt}
V_{l-d}\otimes_k k[[t]]\oplus V_{l-e}\otimes_k k[[t]]=(\pi_{d,e}\circ\iota)^*\cG_{d,e}^l\xrightarrow{(AB+tF, AC+tG)} \iota^*\cF_{d,e}^l,
\end{equation}
where we view $\cF_{d,e}^l$ as a saturated sub-$k[[t]]$-module of $V_l\otimes_kk[[t]]$. By \eqref{soscoh} and right exactness of tensor product, the $k[[t]]$-module $\iota^*\cH_{d,e}^l$ is the cokernel of \eqref{kt}.

We first compute the dimension of the $k$-vector space $(\cH_{d,e}^l)_x=(\iota^*\cH_{d,e}^l)_0$. It is the cokernel of the linear map $V_{l-d}\oplus V_{l-e}\to (\cF_{d,e}^l)_x$ given by $(I,J)\mapsto ABI+ACJ$ (where we view $(\cF_{d,e}^l)_x$ as a subspace of $V_l$). Since $V_{l-d}\oplus V_{l-e}$ and $(\cF_{d,e}^l)_x$ have the same dimension (because $e-1\leq l\leq e+d-1$), the cokernel and the kernel of this linear map have the same dimension. The dimension of the kernel is easy to compute (because $B$ and~$C$ are coprime): it is equal to $0$ if~$u>l-e$ and to~$l-e-u+1$ if $u\leq l-e$.

It remains to show that, in restriction to some neighborhood of $x$, the scheme-theoretic support of $\cH_{d,e}^l$ is included in $\whE_{d,e}^u$. As taking the support commutes with base change, it suffices to show that the support of $\iota^*\cH_{d,e}^l$ is included in the reduced special point of~$\Spec(k[[t]])$. One must prove that if a section of $\iota^*\cH_{d,e}^l$ is killed by~$t^2$, then it is killed by $t$. In turn, it suffices to verify that if~$I\in V_{l-d}\otimes_k k[[t]]$ and~$J\in V_{l-e}\otimes_k k[[t]]$ satisfy $t^2\mid I(AB+tF)+J(AC+tG)$, then~$t\mid I$ and~$t\mid J$. 

Let $I_0,I_1\in V_{l-d}$ and $J_0,J_1\in V_{l-e}$ be the terms of $I$ and $J$ of order~$0$ and $1$ in~$t$. The hypothesis means that $I_0AB+J_0AC=0$ and ${A(I_1B+J_1C)+I_0F+J_0G=0}$. From the first equation and because $B$ and $C$ are coprime, we see that it is possible to write $I_0=CH$ and $J_0=-BH$. Consequently, one sees from the second equation, and from the fact that $A$ and $BG-CF$ are coprime, that $A\mid H$. For degree reasons, one must have $H=0$. It follows that $I_0=J_0=0$, as wanted.
\end{proof}

\subsection{Identifying \texorpdfstring{$\whP_{d,e}$}{hatPde} and \texorpdfstring{$\Hilb_{d,e}$}{Hilbde}}

Recall the definition of $\rho_{d,e}$ given in \S\ref{parboundary}.
 
\begin{thm}
\label{thHilb}
The rational map $\rho_{d,e}:\whP_{d,e}\dashrightarrow \Hilb_{d,e}$ is an isomorphism.
\end{thm}

\begin{proof}
The construction of $\whP_{d,e}$ as a composition of blow-ups in a projective bundle over a projective space shows that the line bundle
$$\cL:=\cO_{\whP_{d,e}}(l,m)\Big(\sum_{u=1}^{d-1}-n_u\whE_{d,e}^u\Big)$$
on $\whP_{d,e}$ is ample for $0<n_{d-1}\ll\dots\ll n_1\ll m\ll l$
 (recall that the pull-back and the strict transform of $E_{d,e}^u$ in $\whP_{d,e}$ are equal because they coincide away from a locus of codimension $\geq 2$). In addition, the line bundle
$$\cM:=\cO_{\Hilb_{d,e}}(l,m)\Big(\sum_{u=1}^{d-1}-n_uD_{d,e}^u\Big)$$
on $\Hilb_{d,e}$ is semi-ample thanks to Propositions \ref{Lisemiample} \ref{Li} and \ref{Licalcul} because it is a linear combination with nonnegative coefficients of the $(\cM_{d,e}^i)_{0\leq i\leq d}$.

By Lemmas \ref{excexc2} and \ref{excexc3}, the rational map $\rho_{d,e}:\whP_{d,e}\dashrightarrow\Hilb_{d,e}$ is an isomorphism in codimension $1$ sending $\whE_{d,e}^u$ to $D_{d,e}^u$. Let $\Theta_{d,e}$ be the 
biggest open subset on which $\rho_{d,e}$ is an isomorphism. One then has $\cL|_{\Theta_{d,e}}\simeq\cM|_{\Theta_{d,e}}$. Note that
\begin{equation}
\label{ProjM}
\whP_{d,e}=\Proj\Big(\bigoplus_{l\geq 0}H^0(\whP_{d,e},\cL^{\otimes l})\Big)=\Proj\Big(\bigoplus_{l\geq 0}H^0(\Theta_{d,e},\cL^{\otimes l}|_{\Theta_{d,e}})\Big)
\end{equation}
by ampleness of $\cL$, and that the contraction induced by $\cM$ is given by
\begin{equation}
\label{ProjL}
\Hilb_{d,e}\to \Proj\Big(\bigoplus_{l\geq 0}H^0(\Hilb_{d,e},\cM^{\otimes l})\Big)=\Proj\Big(\bigoplus_{l\geq 0}H^0(\Theta_{d,e},\cM^{\otimes l}|_{\Theta_{d,e}})\Big).
\end{equation}
Combining \eqref{ProjM} and \eqref{ProjL} yields a morphism $\Hilb_{d,e}\to\whP_{d,e}$ which is the identity on $\Theta_{d,e}$. This contradicts the smoothness of~$\whP_{d,e}$ proved in Theorem \ref{blowup} \ref{bi}  (see \eg \cite[\S1.40]{Debarre}).
\end{proof}

\begin{cor}
\label{coropsi}
For $1\leq u<d<e$, the morphism $\psi_{d,e}^u$ defined in Lemma \ref{excexc1} is a closed immersion.
\end{cor}

\begin{proof}
In the commutative diagram \eqref{diagbexcdivcomm} of Lemma \ref{excexc3}, the left vertical arrow is an isomorphism by Theorem \ref{blowup} \ref{bii} and the horizontal arrows are isomorphisms by Theorem~\ref{thHilb}. It follows that the right vertical arrow is an isomorphism, which exactly means that~$\psi_{d,e}^u$ is a closed immersion.
\end{proof}

\section{Birational models of \texorpdfstring{$P_{d,e}$}{Pde}}
\label{secMMP}

In this section, we study varieties constructed as contractions of~$\whP_{d,e}$.

\subsection{The nef cone of \texorpdfstring{$\whP_{d,e}$}{hatPde}}

The isomorphism $\rho_{d,e}:\whP_{d,e}\isoto\Hilb_{d,e}$ (see Theorem~\ref{thHilb}) is such that $\rho_{d,e}(\whE_{d,e}^u)=D_{d,e}^u$ for $1\leq u<d$ (see Lemma \ref{excexc3}). For~$0\leq i\leq d$, we define $\cL_{d,e}^i:=(\rho_{d,e})^*\cM_{d,e}^i$. By Proposition \ref{Licalcul}, one has 

\begin{equation}
\label{redefLi}
\cL_{d,e}^i\simeq\cO_{\whP_{d,e}}(e-d+i,i)\big(\sum_{u=1}^{i-1}-(i-u)\whE_{d,e}^u\big).
\end{equation}

\begin{lem}
\label{restrmi}
Fix $1\leq u<d<e$ and $0\leq i\leq d$. Let $p_1$ and $p_2$ be the projections of $ \whE_{d,e}^u\simeq \whP_{d-u,e+u}\times\whP_{u,e-d+u}$ onto its factors (see \eqref{canoniso}).
\begin{enumerate}[label=(\roman*)] 
\item
\label{trivi}
The line bundle $\cL_{d,e}^{d}$ is trivial.
\item 
\label{trivii}
If $i\geq u$, then $\cL_{d,e}^i|_{\whE_{d,e}^u}\simeq p_1^*\cL^{i-u}_{d-u,e+u}$. 
\item
\label{triviii}
If $i\leq u$, then $\cL_{d,e}^i|_{\whE_{d,e}^u}\simeq p_1^*\cO_{\whP_{d-u,e+u}}(e-d+2i,0)\otimes p_2^*\cL_{u, e-d+u}^i$.
\end{enumerate}
\end{lem}

\begin{proof} 
Assertion \ref{trivi} is a consequence of Proposition \ref{Lisemiample} \ref{Lii}.
When $i\leq u$, assertion~\ref{triviii} follows at once from Theorem \ref{blowup}~\ref{biii},~\ref{biv},~\ref{bv} and~\ref{bvi}, in view of \eqref{redefLi}. When~${i\geq u}$, the same argument shows that $\cL_{d,e}^i|_{\whE_{d,e}^u}\simeq p_1^*\cL^{i-u}_{d-u,e+u}\otimes p_2^*\cL_{u, e-d+u}^{u}$. Combining this identity with \ref{trivi} proves \ref{trivii}.
\end{proof}

The nef cone $\Nef(\whP_{d,e})$ of $\whP_{d,e}$ is the closed convex cone in $\Pic(\whP_{d,e})\otimes_{\Z}\R$ generated by nef line bundles.
 
\begin{prop}
\label{nef}
Fix $1\leq d<e$. The cone $\Nef(\whP_{d,e})$ is simplicial, and generated by the semi-ample line bundles $(\cL_{d,e}^i)_{0\leq i<d}$.
\end{prop}

\begin{proof}
The case $d=1$ is trivial (see \S\ref{parPicard}). We henceforth assume that $d\geq 2$.

Since $\whP_{d,e}$ has been constructed from $P_{d,e}$ by blowing up $d-2$ times a connected smooth subvariety of codimension $\geq 2$ (see Remark \ref{blowupremarks} (ii)), the Picard group of $\whP_{d,e}$ has rank $d$ and is generated by $\cO_{\whP_{d,e}}(1,0)$, by $\cO_{\whP_{d,e}}(0,1)$ and by the~$(\whE_{d,e}^u)_{1\leq u\leq d-2}$. By \eqref{redefLi}, the $(\cL_{d,e}^i)_{0\leq i<d}$ form a basis of ${\Pic(\whP_{d,e})\otimes_{\Z}\R}$, so they span a simplicial convex cone~$\Sigma\subset \Pic(\whP_{d,e})\otimes_{\Z}\R$. By Proposition \ref{Lisemiample} \ref{Li}, the $(\cL_{d,e}^i)_{0\leq i<d}$ are semi-ample. It follows that~$\Sigma\subset\Nef(\whP_{d,e})$.

We claim that there exist effective curves~$(\Gamma_j)_{0\leq j<d}$ in $\whP_{d,e}$ such that the degree of $\cL^i_{d,e}$ on $\Gamma_j$ is zero if and only if~$i\neq j$. The claim implies the other inclusion~$\Nef(\whP_{d,e})\subset\Sigma$. Indeed, since an element of $\Nef(\whP_{d,e})$ has nonnegative degree on the~$(\Gamma_j)_{0\leq j<d}$, it must belong to $\Sigma$.

We prove the claim by induction on $d$. Since $\whE_{d,e}^{\,1}\simeq \whP_{d-1,e+1}\times \whP_{1,e-d+1}$ by~\eqref{canoniso}, using the computation of $\cL^i_{d,e}|_{\whE_{d,e}^{\; 1}}$ given in Lemma \ref{restrmi} and applying the induction hypothesis to $\whP_{d-1,e+1}$, it is possible to construct all the $\Gamma_j$ except for~$\Gamma_0$ as curves lying on $\whE_{d,e}^{\,1}$. If $d=2$, choose $\Gamma_0$ to be any curve contracted by the natural map~$\whP_{2,e}\isoto P_{2,e}\to P_2$. If $d>2$, choose $\Gamma_0\subset \whE_{d,e}^{\,2}$ to be any curve contracted by the natural map $\whE_{d,e}^{\,2}\stackrel{\eqref{canoniso}}{\simeq}\whP_{d-2,e+2}\times\whP_{2,e-d+2}\to\whP_{d-2,e+2}\times P_2$. That these choices work is again a consequence of the computations of Lemma~\ref{restrmi}.
\end{proof}

\begin{rem}
Lemma \ref{restrmi} \ref{trivi} can be thought of as computing the class in $\Pic(\whP_{d,e})$ of the strict transform $\whE_{d,e}^{d-1}\subset\whP_{d,e}$ of the resultant divisor. One gets
$$\cO_{\whP_{d,e}}(\whE_{d,e}^{d-1})\simeq\cO_{\whP_{d,e}}(e,d)\big(\sum_{u=1}^{d-2}-(d-u)\whE_{d,e}^u\big).$$
This formula recovers the homogeneity degrees $e$ and $d$ of the resultant polynomial.
\end{rem}

\subsection{The birational models \texorpdfstring{$\oP^{\; i}_{d,e}$}{overlinePide} and \texorpdfstring{$P^{\; i}_{d,e}$}{Pide} of \texorpdfstring{$P_{d,e}$}{Pde}}

Let us introduce contractions of $\whP_{d,e}$, whose existence follows from Proposition \ref{nef}. For $0\leq i<d$, let~$\whP_{d,e}\to\oP^{\; i}_{d,e}$ be the contraction induced by the semi-ample line bundle $\cL_{d,e}^i$ on $\whP_{d,e}$. For~${1\leq i<d}$, let~${\whP_{d,e}\to P^{\; i}_{d,e}}$ be the contraction induced by the face of~$\Nef(\whP_{d,e})$ spanned by~$\cL_{d,e}^{i-1}$ and $\cL_{d,e}^i$. By construction, one has
\begin{equation}
\label{projeq}
\begin{alignedat}{2}
&\oP^{\; i}_{d,e}=\Proj\Big(\bigoplus_{l\geq 0}H^0(\whP_{d,e},(\cL_{d,e}^i)^{\otimes l})\Big)\hspace{1em}\textrm{ and }  \\
& P^{\; i}_{d,e}=\Proj\Big(\bigoplus_{l\geq 0}H^0(\whP_{d,e},(\cL_{d,e}^{i-1}\otimes\cL_{d,e}^i)^{\otimes l})\Big),
\end{alignedat}
\end{equation}
so these varieties are projective. For~${1\leq i<d}$, we consider the induced contractions~${\gamma_{d,e}^{\; i}:P^{\; i}_{d,e}\to\oP^{\; i}_{d,e}}$ and~${\delta_{d,e}^{\; i}:P^{\; i}_{d,e}\to\oP^{\; i-1}_{d,e}}$. In particular, one has $P^{\; 1}_{d,e}=P_{d,e}$ and $\oP^{\; 0}_{d,e}=P_d$, and the contraction~${\delta_{d,e}^{\; 1}:P_{d,e}\to P_d}$ is the natural projection.

\subsection{Differentials of morphisms to Grassmannians}

The following two propositions will help us describe the varieties $\oP^{\; i}_{d,e}$ and $P^{\; i}
_{d,e}$ in \S\ref{parstrat}.

\begin{prop}
\label{imm1}
For all $1\leq d\leq l$, the morphism $\eta_d^l:P_d\to \Grass(l-d+1,V_l)$ defined by $\eta_d^l(\lll F\rr)=\lll F\rr_l$ is immersive.
\end{prop}

\begin{proof}
For $\lll F\rr\in P_d$, one computes that $(d\eta_d^l)_{\lll F\rr}:V_d/\lll F\rr\to \Hom(\lll F\rr_l,V_l/\lll F\rr_l)$ is
induced by $F'\mapsto \big(IF\mapsto IF'\big)$. If $F'\in V_d$ is such that $IF'$ is a multiple of $F$ for all $I\in V_{l-d}$, then $F'$ is a multiple of $F$. This shows that  $(d\eta_d^l)_{\lll F\rr}$ is injective.
\end{proof}

\begin{prop}
\label{imm2}
For $1\leq d<e$ and $e\leq l\leq e+d-1$, associating with $[F,G]$ the subspace $\lll F,G\rr_l$ of $V_l$ spanned by multiples of $F$ and $G$ induces a morphism $\theta_{d,e}^l:P_{d,e}\setminus W_{d,e}^{l-e}\to \Grass(2l+2-d-e,V_l)$ which is immersive on~$P_{d,e}\setminus W_{d,e}^{l-e+1}$.
\end{prop}

\begin{proof}
The middle vertical arrow of \eqref{de} is a morphism 
\begin{equation}
\label{subspmor}
\cG_{d,e}^l\to V_l\otimes\cO_{P_{d,e}}
\end{equation}
between locally free sheaves on $P_{d,e}$. The fiber of \eqref{subspmor} at $[F,G]\in P_{d,e}$ can be identified, after splitting the fiber of the first row of~\eqref{de} at $[F,G]$, with the linear map $V_{l-d}\oplus V_{l-e}\to V_l$ given by $(I,J)\mapsto IF+JG$. This linear map is injective if~$[F,G]\notin W_{d,e}^{l-e}$. It follows that the restriction to ${P_{d,e}\setminus W_{d,e}^{l-e}}$ of the cokernel of~\eqref{subspmor} is locally free. We deduce that \eqref{subspmor} induces the required morphism $\theta_{d,e}^l$.

We now fix $[F,G]\in P_{d,e}\setminus W_{d,e}^{l-e+1}$ and show that  $\theta_{d,e}^l$ is immersive at $[F,G]$. Using the rational map $V_d\times V_e\dashrightarrow P_{d,e}$ given by $(F,G)\mapsto [F,G]$, one can identify~$T_{[F,G]}P_{d,e}$ with $V_d/\lll F\rr\oplus V_e/\lll F,G\rr_e$ and compute that
$$(d\theta_{d,e}^l)_{[F,G]}:V_d/\lll F\rr\oplus V_e/\lll F,G\rr_e\to \Hom(\lll F,G\rr_l,V_l/\lll F,G\rr_l)$$
is induced by 
\begin{equation}
\label{diffde}
(F',G')\mapsto \Big(IF+JG\mapsto IF'+JG'\Big).
\end{equation}
Assume that $(F',G')\in V_d\oplus V_e$ is in the kernel of \eqref{diffde}, \ie that $IF'+JG'\in \lll F,G\rr_l$ for all $I\in V_{l-d}$ and $J\in V_{l-e}$. One must show that $F'\in\lll F\rr$ and $G'\in \lll F,G\rr_e$.

Lemma \ref{imm3} below shows that $G'\in \lll F,G\rr_e$. It also shows that  $HF'\in  \lll F,G\rr_e$ for all~$H\in V_{e-d}$. If $\lll F'\rr_e\subset \lll F\rr_e$, then $F'\in \lll F\rr$. Otherwise, the spaces $\lll F'\rr_e$ and $\lll F\rr_e$ are distinct hyperplanes in $\lll F,G\rr_e$. It follows that $M:=\gcd(F,F')$ has degree~$d-1$, and that~$M$ divides $G$. As a consequence, one has $[F,G]\in W^{1}_{d,e}$, which is a contradiction.
\end{proof}

We used the following lemma.

\begin{lem}
\label{imm3}
Fix $d,e,l,u\in\Z$ with $1\leq l-e+1<u<d<e$. Choose $[F,G]$ in~$W_{d,e}^u\setminus W_{d,e}^{u-1}$ and $G'\in V_e$. If ${JG'\in \lll F,G\rr_l}$ for all~$J\in V_{l-e}$, then $G'\in \lll F,G\rr_e$.
\end{lem}

\begin{proof}
Set $A:=\gcd(F,G)\in V_{d-u}$. Write $F=AB$ and~$G=AC$ with $B\in V_{u}$ and~$C\in V_{e-d+u}$. 
The hypothesis implies that~$G'=AC'$ for some $C'\in V_{e-d+u}$ and that~${JC'\in\lll B,C\rr_{l-d+u}}$ for all~${J\in V_{l-e}}$. After adding to $C'$ a multiple of~$C$ (which has the effect of adding to~$G'$ a multiple of $G$), we may assume that~$B$ and~$C'$ are not coprime. Set~${K:=\gcd(B,C')\in V_v}$ with $v>0$. Write $B=KL$ and~$C'=KL'$ for some $L\in V_{u-v}$ and $L'\in V_{e-d+u-v}$. As~$B$ and~$C$ are coprime, one has
\begin{equation}
\label{analysesynthese}
JL'\in\lll L,C\rr_{l-d+u-v}\textrm{ \hspace{.5em}for all\;  }J\in V_{l-e}.
\end{equation}
If $e-d+u>l-d+u-v$, then \eqref{analysesynthese} really means that $JL'\in \lll L\rr_{l-d+u-v}$ for all~$J\in V_{l-e}$, \ie that $G'$ is a multiple of $F$. Assume otherwise. Assertion \eqref{analysesynthese} implies that $\lll L,L'\rr_{l-d+u-v}\subset \lll L,C\rr_{l-d+u-v}$. As $L$ is coprime to both $L'$ and~$C$, we deduce that $(\HF_{u-v,e-d+u}-\HF_{u-v,e-d+u-v})(l-d+u-v)\geq 0$ (by Lemma~\ref{HF}). However, it follows from \eqref{defHF} that the function $\HF_{u-v,e-d+u}-\HF_{u-v,e-d+u-v}$ is nonpositive and, since $u>l-e+1$, that $l-d+u-v$ belongs to the range where it is actually negative. This is a contradiction.
\end{proof}

\subsection{Stratifications of \texorpdfstring{$\oP^{\; i}_{d,e}$}{overlinePide} and \texorpdfstring{$P^{\; i}_{d,e}$}{Pide}}
\label{parstrat}

Let $f:X\to Y$ be a morphism of algebraic varieties. We say that an open subset~$U\subset X$ is \textit{saturated} with respect to~$f$ if~$U=f^{-1}(f(U))$.

\begin{lem}
\label{stratesgen}
Fix $1\leq i<d<e$. The open subset $\whP_{d,e}\setminus\bigcup_{u=1}^{i}\whE_{d,e}^u$ of~$\whP_{d,e}$ is saturated with respect to the contraction $\whP_{d,e}\to\oP^{\; i}_{d,e}$, and its image in~$\oP^{\; i}_{d,e}$ is an open subset naturally isomorphic to $P_{d,e}\setminus W_{d,e}^{i}$.
\end{lem}

\begin{proof}
It suffices to show that an integral curve $\Gamma\subset\whP_{d,e}$ that meets $\whP_{d,e}\setminus\bigcup_{u=1}^{i}\whE_{d,e}^u$ is contracted by $\whP_{d,e}\to\oP^{\; i}_{d,e}$
if and only if it is contracted by $\whP_{d,e}\to P_{d,e}$. If $\Gamma$ is contracted by $\whP_{d,e}\to P_{d,e}$, then since $\cL_{d,e}^i|_{\whP_{d,e}\setminus\bigcup_{u=1}^{i}\whE_{d,e}^u}$ is the pull-back of a line bundle on $P_{d,e}\setminus W_{d,e}^{i}$, it is necessarily contracted by $\whP_{d,e}\to\oP^{\; i}_{d,e}$. 

Conversely, suppose that $\Gamma$ is contracted by $\whP_{d,e}\to\oP^{\; i}_{d,e}$. Then, by definition of $\cL_{d,e}^i$, it is contracted by the morphism~${\chi^{e+i-1}_{d,e}\circ\rho_{d,e}}$ (where~$\chi^{e+i-1}_{d,e}$ is defined in~\eqref{defchi}) induced by $[F,G]\mapsto \lll F,G\rr_{e+i-1}$. By Proposition \ref{imm2}, the restriction~$(\chi^{e+i-1}_{d,e}\circ\rho_{d,e})|_{\whP_{d,e}\setminus\bigcup_{u=1}^{i}\whE_{d,e}^u}$ descends to an immersion defined on $P_{d,e}\setminus W_{d,e}^i$. It follows that, unless $\Gamma$ is included in $\bigcup_{u=1}^{i}\whE_{d,e}^u$, it is contracted by $\whP_{d,e}\to P_{d,e}$.
\end{proof}

\begin{prop}
\label{strates}
Fix $1\leq d<e$.
\begin{enumerate}[label=(\roman*)] 
\item 
\label{strati}
For $0\leq i< d$, the variety $\oP^{\; i}_{d,e}$ has a stratification by $i+1$ locally closed subvarieties with strata isomorphic to $(P_{d-j,e+j}\setminus W_{d-j,e+j}^{i-j})_{0\leq j<i}$ and~$P_{d-i}$.
\item
\label{stratii}
For $1\leq i<d$, the variety $P^{\; i}_{d,e}$ has a stratification by $i$ locally closed subvarieties with strata isomorphic to $(P_{d-j,e+j}\setminus W_{d-j,e+j}^{i-j-1})_{0\leq j<i}$.
\end{enumerate}
\end{prop}

\begin{proof}
We only prove \ref{strati}, as the proof of \ref{stratii} is entirely analogous. If $i=0$, we know that $\oP^{\; 0}_{d,e}\simeq P_d$, so we may suppose that $i\geq 1$.

Lemma \ref{stratesgen} shows that the subset $\whP_{d,e}\setminus\bigcup_{u=1}^{i}\whE_{d,e}^{u}$ of $\whP_{d,e}$ is saturated with respect to ${\whP_{d,e}\to\oP^{\; i}_{d,e}}$, and that its image in $\oP^{\; i}_{d,e}$ is open and isomorphic to~${P_{d,e}\setminus W_{d,e}^{i}}$.

We claim, and prove by induction on $1\leq j<i$,
that $\whE_{d,e}^j\setminus \bigcup_{u=j+1}^{i}(\whE_{d,e}^u|_{\whE_{d,e}^j})$ is saturated with respect to ${\whP_{d,e}\to\oP^{\; i}_{d,e}}$ and that the seminormalization of its image in $\oP^{\; i}_{d,e}$ is isomorphic to $P_{d-j,e+j}\setminus W_{d-j,e+j}^{i-j}$. Using $\whE_{d,e}^j\simeq \whP_{d-j,e+j}\times\whP_{j,e-d+j}$ (see~\eqref{canoniso}) and letting $p_1$ be the first projection, one has $\whE_{d,e}^u|_{\whE_{d,e}^j} =p_1^*\whE^{u-j}_{d-j,e+j}$ (by Theorem \ref{blowup} \ref{biv}) and $\cL^i_{d,e}|_{\whE_{d,e}^j}\simeq \cL^{i-j}_{d-j,e+j}$ (by Lemma~\ref{restrmi}~\ref{trivii}). The claim therefore follows from Lemma \ref{stratesgen} applied to $\whP_{d-j,e+j}\to\oP^{\; i-j}_{d-j,e+j}$.

As $\whE_{d,e}^{i}$ is the complement of all the saturated subsets of $\whP_{d,e}$ already described, it is also saturated. The description of $\cL_{d,e}^i|_{\whE_{d,e}^i}$ given in Lemma \ref{restrmi} \ref{trivii} shows that the seminormalization of its image in $\oP^{\; i}_{d,e}$ is isomorphic to $P_{d-i}$.

It remains to prove that the strata distinct from the open one are isomorphic to~$(P_{d-j,e+j}\setminus W_{d-j,e+j}^{i-j})_{1\leq j<i}$ and~$P_{d-i}$ on the nose (and not only after seminormalization). By definition of~$\oP^{\; i}_{d,e}$, the morphism ${\chi^{e+i-1}_{d,e}\circ\rho_{d,e}}$ (where~$\chi^{e+i-1}_{d,e}$ is defined in~\eqref{defchi}) factors through~$\oP^{\; i}_{d,e}$, yielding a commutative diagram
\begin{equation}
\begin{gathered}
\label{factorization}
\xymatrix@R=1em{
\whP_{d,e}\ar[rd]\ar[rr]^{\chi^{e+i-1}_{d,e}\circ\,\rho_{d,e}\hspace{6em}} & &\Grass\big(\HF_{d,e}(e+i-1),V_{e+i-1}\big). \\
&  \oP^{\; i}_{d,e} \ar[ru]_{\varepsilon_{d,e}^i}&
}
\end{gathered}
\end{equation}

Fix $1\leq j<i$. The seminormalization $P_{d-j,e+j}\setminus W_{d-j,e+j}^{i-j}\to \oP^{\; i}_{d,e}$ of the $j$-th stratum induces after composition with $\varepsilon_{d,e}^i$ a morphism
\begin{equation}
\label{compositiongrass}
P_{d-j,e+j}\setminus W_{d-j,e+j}^{i-j}\to \Grass\big(\HF_{d,e}(e+i-1),V_{e+i-1}\big).
\end{equation}
The morphism \eqref{compositiongrass} coincides with $\theta_{d-j,e+j}^{e+i-1}|_{(P_{d-j,e+j}\setminus W_{d-j,e+j}^{i-j})}$. Indeed, a generic point $([A,H],[B,C])\in \whE_{d,e}^j\subset\whP_{d,e}$ (see \eqref{canoniso}) is sent by $\chi^{e+i-1}_{d,e}\circ\,\rho_{d,e}$ to the subspace of~$V_{e+i-1}$ generated by the multiples of $AB$, of $AC$ and of $H$ (by Lemmas~\ref{excexc1} and \ref{excexc3}), hence equal to $\lll A,H\rr_{e+i-1}$ because $\lll B,C\rr_{e-d+i+j-1}=V_{e-d+i+j-1}$. It follows from Proposition~\ref{imm2} that~\eqref{compositiongrass} is immersive. We deduce that the seminormalization morphism is immersive, hence an isomorphism onto its image.

The argument for the last stratum is similar. Its seminormalization $P_{d-i}\to\oP^{\; i}_{d,e}$ induces after composition with $\varepsilon_{d,e}^i$ the morphism $\eta_{d,e}^{e+i-1}$, which is immersive by Proposition \ref{imm1}. It follows that the seminormalization morphism is immersive, hence an isomorphism onto its image.
\end{proof}

Let us put forward the following corollary of \eqref{projeq} and Proposition \ref{strates} \ref{strati}.

\begin{cor}
\label{corcompactification}
For $1\leq d<e$, the variety $\oP^{\; d-1}_{d,e}$ is a projective compactification of~$U_{d,e}:=P_{d,e}\setminus\Delta_{d,e}$ whose boundary has codimension $2$.
\end{cor}

\subsection{Running the MMP on \texorpdfstring{$P_{d,e}$}{Pde}}
\label{parMMP}

We now show that the varieties~$\oP^{\; i}_{d,e}$ and~$P^{\; i}_{d,e}$ can be used to realize the~MMP for $P_{d,e}$.

\begin{prop}
\label{Qfact}
For $1\leq i<d<e$, the variety $P^{\; i}_{d,e}$ is a small $\Q$-factorial modification of $P_{d,e}$. Its nef cone is generated by two semi-ample line bundles whose pull-backs to $\whP_{d,e}$ are positive multiples of $\cL_{d,e}^{i-1}$ and $\cL_{d,e}^{i}$ respectively.
\end{prop}

\begin{proof}
By Proposition \ref{strates}, the rational map $P_{d,e}\dashrightarrow P^{\; i}_{d,e}$ is an isomorphism in codimension $1$. It therefore induces an isomorphism on class groups. We deduce that~$\Cl(P^{\; i}_{d,e})$ has rank $2$. The pull-backs by~${\delta_{d,e}^{\; i}:P^{\; i}_{d,e}\to\oP^{\; i-1}_{d,e}}$ and~${\gamma_{d,e}^{\; i}:P^{\; i}_{d,e}\to\oP^{\; i}_{d,e}}$ of the ample line bundles on~$\oP^{\; i-1}_{d,e}$ and~$\oP^{\; i}_{d,e}$ stemming from~\eqref{projeq} are line bundles on~$P^{\; i}_{d,e}$ whose pull-backs to $\whP_{d,e}$ are positive multiples of~$\cL_{d,e}^{i-1}$ and $\cL_{d,e}^{i}$ respectively. It follows that $\Pic(P^{\; i}_{d,e})$ has rank $\geq 2$, hence equal to $2$, and that $P^{\; i}_{d,e}$ is $\Q$-factorial.

The two elements of $\Pic(P^{\; i}_{d,e})$ introduced above induce the contractions $\delta_{d,e}^{\; i}$ and~$\gamma_{d,e}^{\; i}$. They are thus semi-ample, and on the boundary of $\Nef(P^{\; i}_{d,e})$. Since $P^{\; i}_{d,e}$ has Picard rank $2$, they generate $\Nef(P^{\; i}_{d,e})$. 
\end{proof}

Here is the main theorem of this article.

\begin{thm}
\label{thMMPmain}
Fix $1\leq d<e$.
\begin{enumerate}[label=(\roman*)] 
\item
\label{maini}
The variety $P_{d,e}$ is a Mori dream space and its effective cone is generated by $\cO_{P_{d,e}}(1,0)$ and $\cO_{P_{d,e}}(\Delta_{d,e})$. 
\item
\label{mainii}
The MMP for $P_{d,e}$ flips the strict transforms of the $W_{d,e}^i$ for $1\leq i\leq d-2$ and eventually contracts the strict transform of $W_{d,e}^{d-1}=\Delta_{d,e}$.
\item
\label{mainiii}
The last model of the MMP for $P_{d,e}$ is a projective compactification of~$U_{d,e}$ with boundary of codimension $2$, that admits a stratification whose strata are isomorphic to $(U_{d-j,e+j})_{0\leq j<d}$.
\end{enumerate}
\end{thm}

\begin{proof}
We may assume that $d>1$. Consider the diagram
\begin{equation}
\begin{gathered}
\label{MMPdiag}
\xymatrix@C=.39em@R=1em{
&P_{d,e}=P_{d,e}^{\; 1}\hspace{1em}\ar[rd]^{\gamma_{d,e}^{\;1}}\ar[ld]_{\delta_{d,e}^{\;1}}&&P_{d,e}^{\; 2}\ar[ld]^{\delta_{d,e}^{\;2}}&\cdots& P_{d,e}^{\; d-2}\ar[rd]_{\gamma_{d,e}^{\;d-2}}&&P_{d,e}^{\; d-1}\ar[ld]_{\delta_{d,e}^{\;d-1}}\ar[rd]^{\gamma_{d,e}^{\;d-1}} \\
P_d=\oP^{\;0}_{d,e}&&  \oP^{\; 1}_{d,e} &&&& \oP^{\;d-2}_{d,e} && \oP^{\;d-1}_{d,e}.
}
\end{gathered}
\end{equation}
 The variety $P_{d,e}$ is the total space of the fibration $\delta_{d,e}^{\; 1}:P_{d,e}\to P_d$. Proposition \ref{strates} implies that~$\gamma_{d,e}^{\; i}$ is a small contraction contracting the strict transform of~$W_{d,e}^i$, and that the small contraction $\delta_{d,e}^{\; i+1}$ is its flip (for $1\leq i\leq d-2$). It also follows from Proposition \ref{strates} that $\gamma_{d,e}^{\; d-1}$ is a divisorial contraction contracting the strict transform of  $W_{d,e}^{d-1}=\Delta_{d,e}$. Diagram \eqref{MMPdiag} implies that the cones $(\Nef(P^{\; i}_{d,e}))_{1\leq i<d}$ cover the movable cone of~$P_{d,e}$. As the $(P^{\; i}_{d,e})_{1\leq i<d}$ are small $\mathbb{Q}$-factorial modifications of~$P_{d,e}$ whose nef cones are spanned by semi-ample line bundles (see Proposition \ref{Qfact}), the hypotheses of \cite[Definition 1.10]{MDS} are satisfied and $P_{d,e}$ is a Mori dream space.

In addition, the existence of the fibration $\delta_{d,e}^{\; 1}$ (resp. of the divisorial contraction~$\gamma_{d,e}^{\; d-1}$) show that $\cO_{P_{d,e}}(1,0)$ (resp. $\cO_{P_{d,e}}(\Delta_{d,e})$) is on the boundary of the effective cone of $P_{d,e}$. As $P_{d,e}$ has Picard rank $2$, they generate it. This proves \ref{maini}.

We have already verified \ref{mainii}. The last model of the MMP for~$P_{d,e}$ is $\oP^{\; d-1}_{d,e}$. Its description in Proposition \ref{strates} \ref{strati} proves \ref{mainiii}.
\end{proof}

\begin{rems}
(i)
Fix $1\leq u<d<e$. Using the description of the $(P^{\; i}_{d,e})_{1\leq i<d}$ given in Proposition~\ref{strates}~\ref{stratii}, one can understand (up to normalization) what happens to~$W_{d,e}^u$ during the MMP for $P_{d,e}$. The normalization of $W_{d,e}^u$ is isomorphic to~${P_{d-u}\times P_{u,e-d+u}}$ (the normalization morphism is the multiplication map $\mu_{d,e}^u$ of~\eqref{defphi}). During the first $u-1$ flips, the variety $W_{d,e}^u$ undergoes the MMP for the second factor $P_{u,e-d+u}$. In particular, after the $(u-1)$-th flip, its normalization becomes isomorphic to $P_{d-u}\times\oP^{\; u-1}_{u, e-d+u}$. During the $u$-th flip, it is contracted via the first projection $P_{d-u}\times\oP^{\; u-1}_{u, e-d+u}\to P_{d-u}$, and then flipped via~$P_{d-u,e+u}\to P_{d-u}$. In the last $d-u-1$ steps, it follows the MMP for $P_{d-u,e+u}$. In particular, in the last model, it gives rise to a subvariety whose normalization is isomorphic to~$\oP^{\; d-u-1}_{d-u,e+u}$.

(ii)
It would be interesting to decide whether $\whP_{d,e}$ itself is a Mori dream space.
\end{rems}

\section{Applications}

\subsection{Projective curves avoiding the resultant divisor}
\label{parcurves}

Using Proposition \ref{strates}, we construct projective curves in $U_{d,e}:=P_{d,e}\setminus\Delta_{d,e}$.

\begin{thm}
\label{completecurves}
The variety $U_{d,e}$ is covered by projective curves.
\end{thm}

\begin{proof}
By Corollary \ref{corcompactification}, the variety $\oP^{\; d-1}_{d,e}$ is a projective compactification of $U_{d,e}$ with codimension $2$ boundary. A general linear section of dimension $1$ of~$\oP^{\; d-1}_{d,e}$ through any given point of $U_{d,e}$ is a projective curve avoiding the boundary.
\end{proof}

\begin{rem}
\label{explicurve}
I do not know how to construct curves as in Theorem \ref{completecurves} directly, without using the existence of the compactification $\oP^{\; d-1}_{d,e}$ of $U_{d,e}$. However, this is possible in some particular cases.

For $e=d+1$, an explicit complete curve in $U_{d,d+1}$ is the closure of the image of
$$t\mapsto[X_0^{d}+tX_0^{d-1}X_1+\dots+t^{d}X_1^{d}, X_0X_1(X_0^{d-1}+tX_0^{d-2}X_1+\dots+t^{d-1}X_1^{d-1})].$$

In positive characteristic $p$, one may also use $p$-th powers. For instance, for~${e\geq 2}$, there is a well-defined map $\psi_p:P_{1,e}\to P_{p,pe}$ given by $\psi_p([F,G])=[F^p,G^p]$. Its image is a complete curve in $U_{p,pe}$.
\end{rem}

\subsection{A finitely generated ring of invariants}
\label{parCox}

Here is an application of Theorem~\ref{thMMPmain} to invariant theory.

\begin{thm}
\label{thmCox}
For $1\leq d<e$, the $k$-algebra of invariants of the action of $V_{e-d}$ on~$V_d\times V_e$ given by  $H\cdot(F,G)=(F,G+HF)$ is finitely generated.
\end{thm}

\begin{proof}
Set $V_d^*:=V_d\setminus\{0\}$. Let $V_{d,e}$ be the total space of the vector bundle on~$V_d^*$ whose fiber over $F$ is $V_e/\lll F\rr_e$. The natural morphism ${V_d^*\times V_e\to V_{d,e}}$ is a Zariski-locally trivial~$V_{e-d}$\,\nobreakdash-torsor on~$V_{d,e}$. As the complement of $V_d^*$ has codimension $\geq 2$ in $V_d$, it follows that
\begin{equation}
\label{invariants}
\cO(V_d\times V_e)^{V_{e-d}}\isoto \cO(V_d^*\times V_e)^{V_{e-d}}\simeq\cO(V_{d,e}).
\end{equation}

Assume first that $d=1$. As the origin has codimension $\geq 2$ in the smooth variety~$V_d$, the line bundle $V_{d,e}$ over $V_d^*$ extends over $V_d$, to a trivial line bundle because~$V_d$ is an affine space. It follows that $\cO(V_{d,e})\simeq\cO(V_d)[t]$ is finitely generated.

Suppose now that $d>1$. Let $V_{d,e}^*$ be the complement of the zero section in~$V_{d,e}$. As $d>1$, the complement of $V_{d,e}^*$ has codimension $\geq 2$ in $V_{d,e}$, so the restriction~map 
\begin{equation}
\label{isofonctions}
\cO(V_{d,e})\isoto \cO(V_{d,e}^*)
\end{equation}
is an isomorphism. The formula $(s,t)\cdot(F,G)=(s F,t G)$ induces a free 
action of~$\G_m^2$ on $V_{d,e}^*$, whose quotient is precisely~$P_{d,e}$. The quotient morphism identifies the space~$H^0(P_{d,e},\cO_{P_{d,e}}(l,m))$ with the subspace of $\cO(V_{d,e})$ on which $(s,t)\in \G_m^2$ acts by multiplication by $s^lt^m$. Summing over all $(l,m)\in\Z^2$ yields an isomorphism
\begin{equation}
\label{isoCox}
\Cox(P_{d,e})\isoto \cO(V_{d,e}^*).
\end{equation}
The $k$-algebra~$\Cox(P_{d,e})$ is finitely generated by \cite[Proposition 2.9]{MDS} and Theorem~\ref{thMMPmain}~\ref{maini}. In view of \eqref{invariants}, \eqref{isofonctions} and \eqref{isoCox}, this proves the theorem.
\end{proof}

\begin{rem}
\label{remCox}
If $d=1$, the $k$-algebras $\Cox(P_{d,e})$ and $\cO(V_d\times V_e)^{V_{e-d}}$ are polynomial rings in respectively $2$ and $3$ variables.
\end{rem}

\bibliographystyle{myamsalpha}
\bibliography{Resultant}

\end{document}